\documentclass{amsart}
\usepackage{amsfonts}
\usepackage{amsmath}
\usepackage{amssymb}
\usepackage{subfig}
\usepackage{paralist}
\usepackage{graphics}
\usepackage{xcolor}
\usepackage{epsfig}
\usepackage[colorlinks=true]{hyperref}

\newcommand{\im}{\text{Im}}

\newcommand{\changetext}[1]{\color{black} #1 \color{black}}
\newcommand{\removetext}[1]{\color{green}  \color{black}}
\newcommand{\mres}{\mathbin{\vrule height 1.6ex depth 0pt width
0.13ex\vrule height 0.13ex depth 0pt width 1.3ex}}

\newcommand{\abs}[1]{\left|{#1}\right|}

\hypersetup{urlcolor=blue, citecolor=red}
\newcommand{\rect}{\underline{\underline{\tau}}}

\numberwithin{equation}{section}
\newtheorem{theorem}{Theorem}[section]
\newtheorem{corollary}[theorem]{Corollary}

\newtheorem{lemma}[theorem]{Lemma}
\newtheorem{proposition}[theorem]{Proposition}
\newtheorem*{problem}{Problem}
\theoremstyle{definition}
\newtheorem{definition}[theorem]{Definition}
\newtheorem{remark}[theorem]{Remark}

\email{qlxia@math.ucdavis.edu}
\email{xushaofeng@ruc.edu.cn}

\begin{document}
\title{Ramified optimal transportation with payoff on the boundary}
\author[Qinglan Xia and Shaofeng Xu]{}
\maketitle

\centerline{\scshape Qinglan Xia } \medskip {\footnotesize %
\centerline{Department of Mathematics}
\centerline{University of California
at Davis} \centerline{Davis, CA 95616, USA} }

\medskip

\centerline{\scshape Shaofeng Xu } \medskip {\footnotesize %
\centerline{School of Economics} \centerline{Renmin University of China} %
\centerline{Beijing, 100872, China} }

\begin{abstract}
This paper studies a variant of ramified/branched optimal transportation
problems. Given the distributions of production capacities and market sizes, a firm looks for an allocation of productions over factories, a distribution of sales across markets, and a transport path that delivers the product to maximize its profit. Mathematically, given any two measures $\mu$ and $\nu$ on $X$, and a payoff function $h$, the planner wants to minimize $\mathbf{M}_{\alpha }(T)-\int_{X}hd(\partial T)$
among all transport paths $T$ from $\tilde{\mu}$ to $\tilde{\nu}$ with $\tilde{\mu}\changetext{\leq} \mu $ and $\tilde{\nu}\changetext{\leq} \nu $, where $\mathbf{M}_{\alpha }$ is the standard cost functional used in ramified transportation. After proving the existence result, we provide a characterization of the boundary measures of the optimal solution. They turn out to be the original measures restricted on some Borel subsets up to a Delta mass on each connected component.
Our analysis further finds that as the boundary payoff increases, the corresponding solution of the current problem converges to an optimal transport path, which is the solution of the standard ramified transportation.
\end{abstract}

\section{Introduction}
\subsection{The ROTPB problem}

Transportation is an important force shaping the spatial distribution of
economic activities. Consider a firm that produces and sells a product in
various regions. Given the locations and capacities of these regions and the
associated production costs and sale prices of the product, the firm looks
for a distribution of productions over factories, a distribution of sales
across markets, and a transport path that delivers the product to maximize
its profit. The firm's optimal plan over productions and sales depends on
its choice of transport path, and vice versa. The interactions between
location and transport choices, however, often render these problems difficult to analyze.

In this paper, we address some of these interactions in the framework of the
ramified optimal transportation. More precisely, we consider the following
resource allocation problem:
Let $\mu $ and $\nu $ be two  \changetext{ Radon} measures on a convex compact subset $X$ of the Euclidean space $\mathbb{R}^m$, $\mathbf{M}_{\alpha }$ be the standard cost functional used in ramified transportation \cite{xia2015motivations} for $\alpha \in [0,1)$ and $%
h$ be a continuous function on the support of the signed measure $\nu -\mu $.  We consider the problem:

\begin{problem}[ROTPB($\protect\mu ,\protect\nu $)]
Minimize
\begin{equation}
\label{eqn: E}
\mathbf{E}_\alpha^h(T):=\mathbf{M}_{\alpha }(T)-\int_{X}hd(\partial T)
\end{equation}%
among all rectifiable 1-current $T$ with $\partial T\preceq \nu -\mu $ as
signed measures.\footnote{The notation $\preceq$ is introduced in (\ref{eqn: measure_preceq}).}
\end{problem}

In the context of the above example, measures $\mu $
and $\nu $ represent, respectively, the distributions of production
capacities and market sizes. The function $h$ represents the payoff associated with moving mass
from $\mu $ to $\nu $, and it captures the
production cost of the product over $\mu $ and its sale price over $\nu $.
The firm aims to maximize its profit defined as sale revenues minus costs
involved in transportation and production. We call this problem as \textit{%
Ramified optimal transportation with payoff on the boundary} (ROTPB).

\subsection{Background}
This paper is related to the literature of optimal transport problems which
concerns efficient mass transportation. These problems are studied early on
by Monge and Kantorovich, and has been extensively analyzed in recent years. Classical references can be found in the books \cite{villani, villani2} by Villani, \cite{santambrogio} by Santambrogio, and the user's guide \cite{guide} by Ambrosio and Gigli. Our paper is most closely related to the ramified optimal transportation (ROT) (also called branched transportation) literature, which models branching transport structures thanks to the efficiency in group transportation. In contrast to the Monge-Kantorovich problems where the transportation cost is solely determined by a transport map, the cost in ramified transport problems is determined by the actual transport path. The Eulerian formulation of the ROT problem is proposed by the first author in \cite{xia1}, with related motivations, set-up and applications surveyed in \cite{xia2015motivations}. An equivalent Lagrangian
formulation of the problem is established by Maddalena, Morel, and Solimini in \cite{msm}. One may refer to \cite{book} for detailed discussions of the research in this direction.  Some interesting recent developments on ROT can be found for instance in \cite{bw, colombo, colombo2020, mmst, psx, ben}.

Our paper differentiates itself
from the existing ROT literature in two main regards.
First, in the literature both measures $\mu $ and $\nu $ are fixed and of
equal mass, and the problem only involves finding a cost-minimizing
transport path. By contrast, the planner in this paper optimizes over all
possible combinations of\ $\left( \tilde{\mu},\tilde{\nu}\right) $ with $%
\tilde{\mu}\changetext{\leq} \mu $, $\tilde{\nu}\changetext{\leq} \nu $ and $||\tilde{\mu}||=||%
\tilde{\nu}||$. \changetext{Similar kind of optimal partial mass transport has been studied for instance by Caffarelli and McCann \cite{Caffarelli_Maccann} and also Figalli \cite{figalli} for the scenario of Monge-Kantorovich problems with a particular attention to the quadratic cost.} Second, the planner faces a reward for relocating mass at
the boundary, and thus the solution relies on the payoff function $h$. This
element has been absent in the literature up to our best knowledge.

\subsection{Main results} Our main results include three parts: the existence theorem (Theorem \ref{thm: existence}), the characterization theorems (Theorem \ref{thm: atomic}, Theorem \ref{main}), and the approximation theorem (Theorem \ref{thm: approximation}).

We first prove the existence of an $\mathbf{E}_\alpha^h$-minimizer $T^*$ for the ROTPB($\protect\mu ,\protect\nu $) problem in Theorem \ref{thm: existence}. This optimal solution $T^*$ is an $\alpha$-optimal transport path \changetext{of finite $\mathbf{M}_\alpha$ cost} from $\mu^*$ to $\nu^*$ for some measures $\mu^*\changetext{\leq} \mu$ and $\nu^*\changetext{\leq} \nu$. As such, $T^*$ automatically inherits many nice geometric properties of optimal transport paths as described previously in \cite{xia2015motivations}.
We next characterize the optimal allocation measures $\mu^*$ and $\nu^*$. In the \changetext{ finite}atomic case, we show

\begin{theorem}
\label{thm: atomic}Suppose $\mu$ and $\nu$ are two \changetext{ finite} atomic measures on $X$, $0<\alpha <1$,
and $T^*\in Path(\mu^*, \nu^*)$ is a solution to the ROTPB($\mu ,\nu $) problem. Let $\{K_k: k=1,2,\cdots, \ell\}$ be the set of the connected components of the support of $T^\ast$. Then, for each $k=1,2,\cdots, \ell$, it holds that
\begin{equation}
\label{eqn: decomposition_component}
\mu ^{\ast } \mres K_k=\mu \mres K_k-m_k\delta_{p_k} \text{ and } \nu ^{\ast } \mres K_k=\nu \mres K_k-n_k \delta_{q_k},
\end{equation}
for some points $p_k\in K_k \cap spt(\mu^\ast)$ and $ q_k\in K_k \cap spt(\nu^\ast)$ with
\[m_k:=\max\{\mu(K_k)-\nu(K_k), 0\} \text{ and } n_k:=\max\{\nu(K_k)-\mu(K_k), 0\}. \]
As a result, we have the decomposition
\begin{equation}
\mu ^{\ast }=\mu \mres A-\mathbf{a} \text{ and } \nu
^{\ast }=\nu \mres B-\mathbf{b},
\end{equation}
for $A=spt(\mu^*), B=spt(\nu^*)$, and
\begin{equation}
\label{eqn: a and b}
\mathbf{a}=\sum_{k=1}^\ell m_k \delta_{p_k},\ \mathbf{b}=\sum_{k=1}^\ell n_k \delta_{q_k}.
\end{equation}
\end{theorem}
Note that in equation (\ref{eqn: decomposition_component}), at least one of $m_k$ and $n_k$ is zero for each $k$. The equation says that on each connected component $K_k$, all existing resources in the optimal allocation source measure $\mu^*$ will be used up, and all demands in the optimal allocation destination measure $\nu^*$ will be met with at most one exception at either a source node or a destination node. There are three scenarios:
\begin{itemize}
    \item In the balanced case where $\mu(K_k)=\nu(K_k)$, then
    \[\mu^*\mres K_k=\mu\mres K_k \text{ and } \nu^*\mres K_k=\nu\mres K_k.\]
    All source and destination nodes are fully in use.
    \item In the over-supply case where  $\mu(K_k)>\nu(K_k)$, then
        \[\mu^*\mres K_k=\mu\mres K_k -(\mu(K_k)-\nu(K_k))\delta_{p_k}\text{ and } \nu^*\mres K_k=\nu\mres K_k.\]
        All source nodes excluding the one at $p_k$ and all destination nodes are fully in use.
 \item In the over-demand case where  $\mu(K_k)<\nu(K_k)$, then
        \[\mu^*\mres K_k=\mu\mres K_k \text{ and } \nu^*\mres K_k=\nu\mres K_k-(\nu(K_k)-\mu(K_k))\delta_{q_k}.\]
        All source nodes and all destination nodes except for the one at $q_k$ are fully in use.
\end{itemize}

In Theorem \ref{main}, we further extend the results of Theorem \ref{thm: atomic} to general cases. 

The third part of the main results highlights an important implication of the current study for solving an optimal transport path. We consider a version of ROTPB problems, where \changetext{the measures $\mu$ and $\nu$ are disjointly supported and} the payoff function $h_c$ takes a constant
value $2c$ on the support of $\nu $, and vanishes on the support of $\mu$. In the early example, the parameter $c$ represents (half of) the gap between the sale price and the production cost, and it effectively determines the payoff from
relocating a unit of mass.  Intuitively, the larger the payoff, the more incentive the planner
has to relocate mass from sources to destinations. When the payoff is sufficiently large, it is in the best interest of the planner to move as much mass as possible. We prove in Theorem \ref{thm: approximation} that an optimal transport path, which solves the standard ramified transportation problem, can be obtained as a limit of the solutions to a sequence of ROTPB problems associated with a series of increasing boundary payoff. This finding thus provides a novel perspective for approximating an optimal transport path.

\section{Preliminaries}
\subsection{Basic notations in geometric measure theory}
We first recall some terminology about rectifiable currents as in \cite%
{fed} or \cite{simon}.

Let $\Omega \subseteq \mathbb{R}^{m}$ be an open domain and for any integer $k\ge 0$ \changetext{let} $\mathcal{D}%
^{k}(\Omega )$ be the set of all $C^{\infty }$ differential $k$-forms in $%
\Omega $ with compact support with the usual Fr\'{e}chet topology \cite{fed}%
. A $k$\textit{-dimensional current }$S$\textit{\ in }$\Omega $\textit{\ }%
is a continuous linear functional on $\mathcal{D}^{k}(\Omega )$. Denote $%
\mathcal{D}_{k}(\Omega )$ as the set of all $k$-dimensional currents in $%
\Omega $. The \textit{mass} of a current $T\in \mathcal{D}_{k}(\Omega )$ is defined by
\[\mathbf{M}(T):=\sup\{\abs{T(\omega)}: ||\omega||\le 1, \omega\in \mathcal{D}^{k}(\Omega )\}.\] Motivated by the Stokes' theorem, the \textit{boundary} of a current $%
S\in \mathcal{D}_{k}(\Omega )$ for $k\ge 1$ is the current $\partial S$ in $\mathcal{D}_{k-1}(\Omega )$ defined by
\begin{equation*}
\partial S\left( \omega \right) :=S\left( d\omega \right)
\end{equation*}%
for any $\omega\in \mathcal{D}^{k-1}(\Omega )$. A current $T\in \mathcal{D}_{k}(\Omega )$ is called \textit{normal} if $\mathbf{M}(T)+\mathbf{M}(\partial T)<+\infty$. A sequence of currents $\left\{S_{i}\right\}$ in $\mathcal{D}_{k}(\Omega )$ is said
to be weakly convergent to another current $S\in \mathcal{D}_{k}(\Omega )$,
denoted by $S_{i}\rightharpoonup S$, if
\begin{equation*}
S_{i}(\omega )\rightarrow S(\omega )
\end{equation*}%
for any $\omega \in \mathcal{D}^{k}(\Omega )$.

As in \cite{simon}, a subset $M\subseteq \mathbb{R}^{m}$ is called
(countably) $k-$\textit{rectifiable} if $M=\bigcup\limits_{i=0}^{\infty
}M_{i}$, where $\mathcal{H}^{k}\left( M_{0}\right) =0$ under the $k-$%
dimensional Hausdorff measure $\mathcal{H}^{k}$ and each $M_{i}$, for $%
i=1,2,\cdots ,$ is a subset of an $k-$dimensional $C^{1}$ submanifold in $%
\mathbb{R}^{m}$. A \textit{rectifiable k-current} $S$ is a $k-$dimensional current coming from
an oriented $k-$rectifiable set with multiplicities. More precisely, $S\in
\mathcal{D}_{k}(\Omega )$ is a \textit{rectifiable k-current }if it can be
expressed as
\begin{equation*}
S\left( \omega \right) =\int_{M} \langle\omega \left( x\right) ,\xi \left( x\right)
\rangle \theta\left(x\right) d\mathcal{H}^{k}\left( x\right) ,\mathit{\ }\forall
\omega \in \mathcal{D}^{k}(\Omega )
\end{equation*}%
where

\begin{itemize}
\item $M$ is an $\mathcal{H}^{k}$ measurable and $k-$rectifiable subset of $%
\Omega $.

\item $\theta $ is an $\mathcal{H}^{k}\mres M$ integrable positive
function and is called the multiplicity function of $S$.

\item $\xi :M\rightarrow \Lambda _{k}\left( \mathbb{R}^{m}\right) $ is an $%
\mathcal{H}^{k}$ measurable unit tangent vector field on $M$ and is called the
orientation of $S$.
\end{itemize}

The rectifiable current $S$ described as above is often denoted by
\begin{equation*}
S=\underset{=}{\tau }(M,\theta ,\xi ).
\end{equation*}%
In this case, the mass of $S$ is expressed as
\[\mathbf{M}(S)=\int_M \theta(x)d\mathcal{H}^k(x).\]
\changetext{Since $\theta$ is $\mathcal{H}^{k}\mres M$ integrable,
each rectifiable current $S$ here is assumed to have finite mass.}
\subsection{Basic notations in ramified optimal transportation}
Let $X$ be a convex compact subset of the Euclidean space $\mathbb{R}^m$.
The ramified optimal transport problem (also called branched optimal transportation problem in the literature) considers the following  Plateau-type problem:
\begin{problem}[ROT]
Given two (positive) measures $\mu^+$ and $\mu^-$ on $X$ of equal mass and $\alpha< 1$, minimize
\begin{equation*}
\mathbf{M}_{\alpha}(T):=\int_M \theta^\alpha d \mathcal{H}^1.
\end{equation*}
among all rectifiable 1-current $T=\underline{\underline{\tau}}(M, \theta, \xi)$ in $\mathbb{R}^m$ with $\partial T=\mu^- - \mu^+$ in the sense of distributions.
\end{problem}
Each rectifiable 1-current $T=\underline{\underline{\tau}}(M, \theta, \xi)$ such that $\partial T=\mu^- - \mu^+$ is called a {\it transport path} from $\mu^+$ to $\mu^-$. Let
\[Path(\mu^+,\mu^-)=\{T \text{ is a rectifiable 1-current}: \partial T=\mu^- - \mu^+\}\]
be the collection of all transport paths from $\mu^+$ to $\mu^-$.

For the ROT problem, the existence of
an $\mathbf{M}_{\alpha}$-minimizer in $Path(\mu^+, \mu^-)$ is shown in \cite%
{xia1}. Each $\mathbf{M}_{\alpha}$-minimizer is called an \textit{$\alpha$-optimal transport path}.  \changetext{One shall note that for some combinations of exponent $\alpha$ and pair of measures $\mu^{\pm}$, it is possible the $\mathbf{M}_\alpha$ cost of any transport path $T\in Path(\mu^+, \mu^-)$ is infinite, and thus the existence of a solution to the ROT problem is trivial in that case.} 

\changetext{
When $1-\frac{1}{m}<\alpha<1$, it is shown in \cite{xia1} that for any pair of measures $\mu^{\pm}$ of equal (finite) mass there exists an $\alpha$-optimal
transport path of finite $\mathbf{M}_{\alpha}$-cost from $\mu^+$ to $\mu^-$. Moreover}, 
a distance is defined by setting
\begin{equation}  \label{d_a distance}
d_{\alpha}(\mu^+, \mu^-):=\min\{\mathbf{M}_{\alpha}(T): \partial T=\mu^- -
\mu^+ \}
\end{equation}
between $\mu^+$ and $\mu^-$. By \cite[Theorem 3.1]{xia1}, it holds that
\begin{equation}
\label{eqn: d_alpha_upperbound}
    d_{\alpha}(\mu^+,\mu^-)\le C_{m,\alpha}  diam(X)||\mu^+||^\alpha,
\end{equation}
where the constant
\begin{equation}
    \label{eqn: C_m_alpha}
    C_{m,\alpha}=\frac{\sqrt{m}}{2(2^{1-m(1-\alpha)}-1)}.
\end{equation}

\changetext{In general, the existence of finite cost $\alpha$-optimal
transport path between $\mu^+$ and $\mu^-$ depends on the dimensional information of the measures (see \cite{ds}, \cite{xia_dimension}).  In \cite{xia_dimension}, the $d_{\alpha}$-metric is defined on the space of finite atomic probability measures for any real number $\alpha<1$.} 

The following notations are also employed in the analysis:

\begin{itemize}
\item Let $\mu$ and $\nu$ be two (positive) measures on $X$.
We say $ \mu \changetext{\leq} \nu$ if $\nu-\mu$ is still a (positive) measure on $X$. In this case, we say that $\mu$ is \textit{feasible} relative to $\nu$.

\item Let $\mu_1=\mu_1^+-\mu_1^-$ and $\mu_2=\mu_2^+-\mu_2^-$ be the Jordan
decompositions of two signed measures. We say
\changetext{
\begin{equation}
\label{eqn: measure_preceq}
    \mu_1 \preceq \mu_2
\end{equation}}
if $%
\mu_1^+ \changetext{\leq} \mu_2^+$ and $\mu_1^-\changetext{\leq} \mu_2^-$.

\item \changetext{For any signed measure $\bar{\mu}$, let $spt(\bar{\mu})$ denote its support, $||\bar{\mu}||$ denote its total variation, and $|\bar{\mu}|$ denote its total variation measure.}
\end{itemize}

For each rectifiable 1-current $T$, its boundary $\partial T$ can be viewed
as a signed measure.

\section{The ROTPB problem}

This section analyzes the spatial resource allocation problem ROTPB($\protect\mu ,\protect\nu $) as stated in Introduction. For simplicity, when both the parameter $\alpha$ and the function $h$ are clear from the context, we simply write $\mathbf{E}_\alpha^h$ given in (\ref{eqn: E}) as $\mathbf{E}$.

The ROTPB($\protect\mu ,\protect\nu $) problem is indeed a double-minimizing problem
\begin{equation*}
\min \left\{ \min \left\{ \mathbf{E}_\alpha^h(T):\partial T=\tilde{\nu}-\tilde{\mu}\right\} :\ \tilde{\mu}\changetext{\leq} \mu ,%
\tilde{\nu}\changetext{\leq} \nu \text{ with }||\tilde{\mu}||=||\tilde{\nu}||\right\} .
\end{equation*}%
For each fixed $\tilde{\mu}\changetext{\leq} \mu , \tilde{\nu}\changetext{\leq} \nu$ with $||\tilde{\mu}||=||\tilde{\nu}||$,
the inner minimization problem
\begin{equation*}
 \min \left\{\mathbf{E}_\alpha^h(T)= \mathbf{M}_{\alpha }(T)-\int_{X}hd(\partial
T):\partial T=\tilde{\nu}-\tilde{\mu}\right\}
\end{equation*}
can be re-written as
\begin{equation}
\label{eqn: inner_minimizing}
 \min \left\{ \mathbf{M}_{\alpha }(T):\ \partial T=\tilde{\nu}-\tilde{\mu}\right\}-\int_{X}hd\tilde{\nu}+\int_{X}hd\tilde{\mu}.
\end{equation}
Thus, under the $%
d_{\alpha }$ metric as given in (\ref{d_a distance}), the \changetext{ ROTPB($\mu ,\nu$%
)} problem can also be expressed as: Minimize
\begin{equation*}
\mathbf{E}(\tilde{\mu},\tilde{\nu}):=d_{\alpha }(\tilde{\mu},\tilde{\nu}%
)-\int_{X}hd\tilde{\nu}+\int_{X}hd\tilde{\mu}
\end{equation*}%
among all feasible measures $\tilde{\mu}\changetext{\leq} \mu $ and $\tilde{\nu}\changetext{\leq} \nu $ with $||%
\tilde{\mu}||=||\tilde{\nu}||$.

From the perspective of the firm in the example given in Introduction, the ROTPB($\mu ,\nu $) problem can be interpreted as follows. Given the distributions
of production capacities ($\mu $) and market sizes ($\nu $), the firm
chooses an operation plan $\tilde{\mu}\changetext{\leq} \mu $ and $\tilde{\nu}\changetext{\leq}
\nu $ to minimize the total costs incurred in production ($\int_{X}hd\tilde{%
\mu}$) and transportation ($d_{\alpha }(\tilde{\mu},\tilde{\nu})$) net the
sale revenue ($\int_{X}hd\tilde{\nu}$).

We now state the existence theorem for the ROTPB($\mu, \nu$) problem.

\begin{theorem}[Existence]\label{thm: existence}
Let $\mu $ and $\nu $ be two \changetext{Radon} measures on $X$, 
$0\le\alpha<1$ and $h$ be a continuous function on the support of the signed measure $\nu
-\mu $.
\changetext{Then there exists a rectifiable 1-current $T^{\ast }$ of finite $\mathbf{M}_\alpha$ cost} that minimizes
\begin{equation*}
\mathbf{E}_\alpha^h(T):=\mathbf{M}_{\alpha }(T)-\int_{X}hd(\partial T)
\end{equation*}%
among all rectifiable 1-current $T$ with $\partial T\preceq \nu -\mu $ as
signed measures.
\end{theorem}
\changetext{
\begin{remark}
By the Jordan decomposition theorem, for any signed measure $\bar{\mu}$, there exists a unique positive measures $\mu^+$ and $\mu^-$ such that $\bar{\mu}=\mu^+-\mu^-$ and $\mu^+\perp \mu^-$. Thus, without loss of generality, we may assume that $\mu$ and $\nu$  are mutually singular when studying the ROTPB($\mu, \nu$) problem.
\end{remark}}
\begin{proof}
We prove this result by using the direct method of calculus of variations. Let $%
\{T_{i}\}$ be any $\mathbf{E}$-minimizing sequence of rectifiable
1-currents. That is,
\begin{equation*}
\lim_{i\rightarrow \infty }\mathbf{E}(T_{i})=\inf \{\mathbf{E}(T):\partial
T\preceq \nu -\mu \},
\end{equation*}%
and $\partial T_{i}\preceq \nu -\mu $ for each $i$.
\changetext{With no loss of generality, we may assume $\mathbf{E}(T_{i})\leq \mathbf{E}(0)=0$.}
Thus,
\begin{equation}
\label{eqn: M_alpha_bound}
\mathbf{M}_{\alpha }(T_{i})=\mathbf{E}(T_{i})+\int_{X}hd(\partial T_{i})\leq
\int_{X}hd(\partial T_{i})\leq \int_{X}|h|d(|\nu-\mu|)<\infty
\end{equation}%
\changetext{
as $h$ is continuous on the compact set $spt(\nu-\mu)$, the support of $\nu-\mu$.
Now suppose $T_{i}\in Path(\tilde{\mu}_i,\tilde{\nu}_i)$. Since $\mathbf{M}_{\alpha }(T_{i})$ is finite, there exists an $\mathbf{M}_{\alpha }$-minimizer with finite cost for the minimization problem
\[\min \left\{ \mathbf{M}_{\alpha }(T):\ \partial T=\partial T_i\right\}.\]
Note this minimizer is also an $\mathbf{E}$-minimizer for the inner minimization problem
(\ref{eqn: inner_minimizing}) with $\tilde{\mu}=\tilde{\mu}_i \text{ and } \tilde{\nu}=\tilde{\nu}_i$. Without loss of generality, we may assume that $T_i$ is such a minimizer, which is an $\alpha$-optimal transport path of finite cost.
By (\ref{eqn: M_alpha_bound}), the sequence $\{\mathbf{M}_{\alpha }(T_{i})\}$ is bounded.
Employing Lemma \ref{lemma: mass_bound} below shows that the sequence $\{\mathbf{M}(T_{i})\}$
is also bounded. As a result, we get a sequence of normal 1-currents $\{T_i\}$ with equi-bounded mass and boundary mass. By the compactness of normal 1-currents (\cite{fed}), and taking a subsequence if necessary, we may assume that the sequence $\{T_{i}\}$ converges to a normal 1-current $T^{\ast }$ with respect to flat convergence. Since $\mathbf{M}_{\alpha }$ is lower
semi-continuous with respect to flat convergence (\cite{crms, mw}), we have
\begin{equation*}
\mathbf{M}_{\alpha }(T^{\ast })\leq \liminf_{i\rightarrow \infty }\mathbf{M}%
_{\alpha }(T_{i})<\infty.
\end{equation*}
According to the rectifiability theorem (e.g., Theorem 2.7 in \cite{xia2}), finite mass and finite $\mathbf{M}_\alpha$ mass together imply that $T^*$ is also 1-rectifiable.  Since $\{T_{i}\}$ converges to $T^{\ast }$ in flat convergence, the sequence $\{\partial T_{i}\}$ is weak-* convergent to $\partial T^{\ast} $ as signed measures.
}

Since $h$ is continuous on the support $spt(\nu -\mu )$, $spt(\partial
T_{i})\subseteq spt(\nu -\mu )$, and $\partial T_{i}$ is weak-$*$ convergent
to $\partial T^{\ast }$, we have
\begin{equation*}
\int_{X}hd(\partial T)=\lim_{i\rightarrow \infty }\int_{X}hd(\partial T_{i}).
\end{equation*}%
As a result,
\begin{equation*}
\mathbf{E}(T^{\ast })=\mathbf{M}_{\alpha }(T^{\ast })-\int_{X}hd(\partial
T^{\ast })\leq \liminf_{i\rightarrow \infty }\{\mathbf{M}_{\alpha
}(T_{i})-\int_{X}hd(\partial T_{i})\}=\lim_{i\rightarrow \infty }\mathbf{E}%
(T_{i}).
\end{equation*}%
When each $\partial T_{i}\preceq \nu -\mu $, its limit $\partial T^{\ast
}\preceq \nu -\mu $  holds as well. This shows that $T^{\ast }$ is a solution to
the ROTPB($\mu ,\nu $) problem.
\end{proof}

The proof of the theorem takes advantage of the following lemma:

\begin{lemma}
\label{lemma: mass_bound} Suppose $T$ is an $\alpha$-optimal transport path \changetext{with $\mathbf{M}_\alpha(T)<\infty$},
then
\begin{equation}  \label{eq: mass_bound}
\mathbf{M}(T)\le \left(\frac{\mathbf{M}(\partial T)}{2}\right)^{1-\alpha}
\mathbf{M}_{\alpha}(T).
\end{equation}
\end{lemma}

\begin{proof}
Suppose $T=\underline{\underline{\tau}}(M, \theta, \xi)$ is an $\alpha$%
-optimal transport path from $\mu^+$ to $\mu^-$, where $\partial T=\mu^-
-\mu^+$ is the Jordan decomposition of $\partial T$ as a signed measure.
Since $T$ is an $\alpha$-optimal transport path \changetext{of finite cost, it follows (from (\ref{eqn: theta}) for instance)} that $\theta(x)\le \mu^+(X)=\frac{1}{2}\mathbf{M}(\partial T)$ for \changetext{$\mathcal{H}^1$-}
a.e. $x\in M$. Thus,
\begin{eqnarray*}
\mathbf{M}(T)&=&\int_M \theta(x)d \mathcal{H}^1(x)=\int_M \theta(x)^{\alpha}
\theta(x)^{1-\alpha} d \mathcal{H}^1(x) \\
&\le & \int_M \theta(x)^\alpha (\mu^+(X))^{1-\alpha}d \mathcal{H}^1(x)=\left(\frac{%
\mathbf{M}(\partial T)}{2}\right)^{1-\alpha} \mathbf{M}_{\alpha}(T).
\end{eqnarray*}
\end{proof}

\changetext{In the rest of the analysis, we assume that $\mu$ and $\nu$ are mutually singular, and $h$ is continuous on the support of $\nu-\mu$.}

\begin{proposition}
\label{prop: Tzero}If $\min \{h(x):x\in spt(\mu )\}\geq \max \{h(x):x\in
spt(\nu )\}$, then $T^\ast=0$ is the unique solution to the ROTPB($\mu ,\nu $)
problem.
\end{proposition}

\begin{proof}
Suppose $T^\ast$ is a solution to the ROTPB($\mu ,\nu $) problem with $\partial T^\ast=\tilde{\nu}-\tilde{%
\mu}$. Since $\tilde{\mu}$ and $\tilde{\nu}$ have the same mass,
\begin{eqnarray*}
&&\mathbf{E}(T^\ast)=\mathbf{M}_{\alpha }(T^\ast)-\int_{X}hd\tilde{\nu}+\int_{X}hd%
\tilde{\mu} \\
&\geq &\mathbf{M}_{\alpha }(T^\ast)-\int_{X}\max \{h(x):x\in spt(\nu )\}d\tilde{%
\nu}+\int_{X}\min \{h(x):x\in spt(\mu )\}d\tilde{\mu} \\
&=&\mathbf{M}_{\alpha }(T^\ast)+\left( \min \{h(x):x\in spt(\mu )\}-\max
\{h(x):x\in spt(\nu )\}\right) \tilde{\mu}(X)\geq 0
\end{eqnarray*}%
\changetext{where} the equality holds if and only if $T^\ast=0$.
\end{proof}

The condition in the proposition implies that it is impossible to obtain
positive net payoff from relocating mass, needless to mention the incurred transportation cost. It is thus in the best interest of
the planner to not move any mass at all. This proposition illustrates the
role of boundary payoff played in the problem, which we will further examine
in Section \ref{sec: impact of h}.

Suppose that the ROTPB($\mu, \nu$) problem has a solution $T^*\in
Path(\mu ^*,\nu^*)$. Then, $T^*$ is inherently an $\alpha$-optimal transport path in $%
Path(\mu ^*,\nu^*)$ \changetext{with finite $\mathbf{M}_\alpha$ cost.} Thus, $T^*$ itself exhibits some nice regularity
properties (acyclic, \changetext{uniform upper-bound on the degree of vertices, uniform lower-bound on the angles between edges at each vertex,} boundary and interior
regularity, etc) as stated in \cite{xia2015motivations}
for being $\mathbf{M}_\alpha$ optimal.

\section{Properties of the optimal allocation measures}
This section is devoted to characterizing the optimal allocation measures $\mu^\ast$ and $\nu^\ast$.
Let
\begin{equation*}
\mathcal{E}(\mu,\nu):=\min\left\{d_{\alpha }(\tilde{\mu},\tilde{\nu}%
)-\int_{X}hd\tilde{\nu}+\int_{X}hd\tilde{\mu}\ \bigg|\
 \tilde{\mu}\changetext{\leq} \mu \text{ and } \tilde{\nu}\changetext{\leq} \nu \text{ with } ||%
\tilde{\mu}||=||\tilde{\nu}||
\right\}
\end{equation*}%
denote the minimum value of the ROTPB($\mu, \nu$) problem. We first observe some basic properties of $\mathcal{E}$.

\begin{proposition}
\label{Prop: E0} Suppose $0 \changetext{\leq} \tilde{\mu}\changetext{\leq} \mu$ and $0\le \tilde{%
\nu}\changetext{\leq} \nu$. Then,
\begin{equation}
\label{eqn: E_monotonic}
0 \ge \mathcal{E}(\tilde{\mu},\tilde{\nu})\ge \mathcal{E}(\mu,\nu).
\end{equation}
In particular, if $\mathcal{E}(\mu,\nu)=0$, then for all $(\tilde{\mu},
\tilde{\nu})$ with $0 \changetext{\leq} \tilde{\mu}\changetext{\leq} \mu$ and $0\le \tilde{\nu}%
\changetext{\leq} \nu$, it holds that $\mathcal{E}(\tilde{\mu},\tilde{\nu})=0$.
\end{proposition}

\begin{proof}
The results follow from the definition of $\mathcal{E}(\mu, \nu)$.
\end{proof}

Here, $\mathcal{E}(\mu,\nu)$ is non-positive and monotonic since $-\mathcal{E}(\mu,\nu)$ represents the overall possible profit generated for the planner from the pair $(\mu, \nu)$.  When $\mathcal{E}(\mu,\nu)=0$, there is no way to generate a non-zero $\mathcal{E}(\tilde{\mu},\tilde{\nu})$ from some part $(\tilde{\mu},\tilde{\nu})$  of  $(\mu, \nu)$.

\begin{proposition}
\label{Prop: E1} Suppose for each $i=1,2$, $T_i^*$ is a solution to the
ROTPB($\mu_i, \nu_i$) problem, and $T_{1+2}^*$ is a solution to the ROTPB($%
\mu_1+\mu_2, \nu_1+\nu_2 $) problem, then
\begin{equation}  \label{eqn: E_T}
\mathbf{E}(T_{1+2}^*)\le \mathbf{E}(T_1^*)+\mathbf{E}(T_2^*).
\end{equation}
\end{proposition}

This proposition implies that
\begin{equation}
\label{eqn: E_subadditivity}
\mathcal{E}(\mu_1+\mu_2,\nu_1+\nu_2)\le \mathcal{E}(\mu_1,\nu_1)+\mathcal{E}%
(\mu_2,\nu_2).
\end{equation}

\begin{proof}
By assumption, for each $i=1,2$, $T_i^*\in Path(\mu^*_i, \nu^*_i)$ with $%
\mu_i^* \changetext{\leq}\mu_i$ and $\nu_i^*\changetext{\leq}\nu_i$. Then, $T_1^*+T_2^*\in
Path(\mu^*_1+\mu^*_2, \nu^*_1+\nu^*_2)$ with $\mu_1^*+\mu_2^*\changetext{\leq}
\mu_1+\mu_2$ and $\nu_1^*+\nu_2^*\changetext{\leq}\nu_1+\nu_2$. Since $T_{1+2}^*$ is a
solution to the ROTPB($\mu_1+\mu_2, \nu_1+\nu_2$) problem, we have
\begin{eqnarray*}
\mathbf{E}(T_{1+2}^*)&\le & \mathbf{E}(T_1^*+T_2^*)=\mathbf{M}%
_{\alpha}(T_1^*+T_2^*)-\int_X h d(\partial T_1^*+\partial T_2^*) \\
&\le & \mathbf{M}_{\alpha}(T_1^*)+\mathbf{M}_{\alpha}(T_2^*)-\int_X h
d(\partial T_1^*)-\int_X h d(\partial T_2^*) \\
&=&\mathbf{E}(T_1^*)+\mathbf{E}(T_2^*).
\end{eqnarray*}
\end{proof}

Following from the above proof, if the equality in (\ref{eqn: E_T}) holds,
then
\begin{equation*}
\mathbf{M}_{\alpha}(T_1^*+T_2^*)=\mathbf{M}_{\alpha}(T_1^*)+\mathbf{M}%
_{\alpha}(T_2^*).
\end{equation*}
\changetext{ Suppose $T_i=\underline{\underline{\tau}}(M_i, \theta_i, \xi_i)$ with $\theta_i(x)>0$ for $\mathcal{H}^1$-a.e. $x\in M_i$ with $i=1,2$.} Since \changetext{$\alpha< 1$},
\begin{eqnarray*}
&&\mathbf{M}_{\alpha}(T_1^*+T_2^*)-\mathbf{M}_{\alpha}(T_1^*)-\mathbf{M}%
_{\alpha}(T_2^*) \\
&&\le \int_{M_1\cap
M_2}(\theta_1(x)+\theta_2(x))^{\alpha}-\theta_1(x)^{\alpha}-\theta_2(x)^{%
\alpha}d\mathcal{H}^1(x)\le 0,
\end{eqnarray*}
where the equalities hold only if $\mathcal{H}^1(M_1\cap M_2)=0.$ 

We now give a necessary condition on the solution to the ROTPB($\mu, \nu$) problem.

\begin{corollary}
\label{cor: zero_energy}
Suppose that $T^*\in
Path(\mu ^*,\nu^*)$ is a solution to the ROTPB($\mu, \nu$) problem. Then $\mathcal{E}(\mu-\mu^*, \nu-\nu^*)=0.$
\end{corollary}

\begin{proof}
Since $T^*\in Path(\mu ^*,\nu^*)$ is a solution to the ROTPB($\mu, \nu$)
problem, $\mathcal{E}(\mu, \nu)=\mathcal{E}(\mu^*, \nu^*) $. By (\ref{eqn: E_monotonic}) and (\ref{eqn: E_subadditivity}),
\begin{equation*}
0\ge \mathcal{E}(\mu-\mu^*, \nu-\nu^*)\ge \mathcal{E}(\mu,\nu)-\mathcal{E}%
(\mu^*,\nu^*)=0.
\end{equation*}
Therefore, $\mathcal{E}(\mu-\mu^*, \nu-\nu^*)=0$.
\end{proof}

The lemma says that the mass left unmoved by the solution would not generate further gains for the planner.

\begin{proposition}
\label{prop: no_sigma_proportion}
Suppose that the ROTPB($\mu, \nu$) problem has a non-zero solution $T^*\in
Path(\mu ^*,\nu^*)$ and $\alpha<1$. Then there exists no real number $%
\sigma>1$ such that $\sigma \mu^*\changetext{\leq} \mu$ and $\sigma \nu^*\changetext{\leq} \nu$.
\end{proposition}

\begin{proof}
Otherwise, assume that there exists a real number $\sigma >1$ such that $%
\sigma \mu ^{\ast }\changetext{\leq} \mu $ and $\sigma \nu ^{\ast }\changetext{\leq} \nu $. We
consider the function
\begin{equation*}
g(\lambda ):=\mathbf{E}(\lambda \mu ^{\ast },\lambda \nu ^{\ast })=\lambda
^{\alpha }d_{\alpha }(\mu ^{\ast },\nu ^{\ast })-\lambda \int_{X}|h|d\nu
^{\ast }+\lambda \int_{X}|h|d\mu ^{\ast }
\end{equation*}%
for $\lambda \in \lbrack 0,\sigma ]$. Since $T^{\ast }\in Path(\mu ^{\ast
},\nu ^{\ast })$ is a non-zero solution to the ROTPB($\mu ,\nu $) problem, $%
d_{\alpha }(\mu ^{\ast },\nu ^{\ast })=\mathbf{M}_{\alpha }(T^{\ast })>0$.
Thus, given $\alpha <1$, 
\begin{eqnarray*}
g^{\prime }(1) &=&\alpha d_{\alpha }(\mu ^{\ast },\nu ^{\ast
})-\int_{X}|h|d\nu ^{\ast }+\int_{X}|h|d\mu ^{\ast } \\
&<&d_{\alpha }(\mu ^{\ast },\nu ^{\ast })-\int_{X}|h|d\nu ^{\ast
}+\int_{X}|h|d\mu ^{\ast } \\
&=&\mathbf{E}(T^{\ast })\leq \mathbf{E}(0)=0.
\end{eqnarray*}%
As a result, there exists a $\lambda ^{\ast }\in (1,\sigma )$ such that $%
g(\lambda ^{\ast })<g(1)$. Because $\sigma \mu ^{\ast }\changetext{\leq} \mu $ and $%
\sigma \nu ^{\ast }\changetext{\leq} \nu $, we also have $\lambda ^{\ast }\mu ^{\ast
}\changetext{\leq} \sigma \mu ^{\ast }\changetext{\leq} \mu $ and $\lambda ^{\ast }\nu ^{\ast
}\changetext{\leq} \sigma \nu ^{\ast }\changetext{\leq} \nu $.
Hence $\mathbf{E}(\lambda ^{\ast }\mu ^{\ast },\lambda ^{\ast }\nu ^{\ast
})=g(\lambda ^{\ast })<g(1)=\mathbf{E}(\mu ^{\ast },\nu ^{\ast })$, which
contradicts with $T^{\ast }$ being a solution to the ROTPB($\mu ,\nu $)
problem.
\end{proof}

At a solution to the ROTPB($\mu ,\nu $) problem, the planner might only move
out a portion of the mass held at one source or ship in mass less than
registered at a single destination. However, the above proposition shows
that this can not happen at all the involved sources and destinations.
Otherwise, an improvement can be achieved by a proportional increase of the
transported mass at these locations. This is because the resulting marginal
payoff from moving more mass outweighs the marginal transportation cost
thanks to the transport economy of scale when $\alpha <1$.

The remainder of this section focuses on characterizing the optimal allocation measures $\mu^\ast$ and $\nu^\ast$, with the main result stated in Theorem \ref{main}. We first set up some technical bases.

\begin{definition}
\label{def: S_on_T}
Let $T=\underline{\underline{\tau}}(M,\theta,\xi)$ and $S=\underline{%
\underline{\tau}}(N,\rho,\eta)$ be two rectifiable 1-currents. We say $S$ is
on $T$ if $\mathcal{H}^1(N\setminus M)=0$, \changetext{and $\rho(x)\le \theta(x)$ for $\mathcal{H}^1$ almost all $ x\in N$. } 
\end{definition}
\changetext{ Note that when $S=\underline{%
\underline{\tau}}(N,\rho,\eta)$ is
on $T=\underline{\underline{\tau}}(M,\theta,\xi)$, then $\xi(x)=\pm \eta(x)$ for $\mathcal{H}^1$
almost all $ x\in N$, since two rectifiable sets have the same tangent a.e. on their intersection.}
\begin{theorem}
\label{thm: S is zero} Suppose that $T^*\in Path(\mu^*, \nu^*)$ is a
solution to the ROTPB($\mu, \nu$) problem, and $0<\alpha<1$. If there exists a rectifiable
1-current $S$ on $T^\ast$ with
\[
\partial(T^\ast+S) \preceq \nu-\mu \text{ and } \partial(T^\ast- S) \preceq \nu-\mu,\]
then $S=0$.
\end{theorem}

\begin{proof}
Assume that $S=\underline{\underline{\tau }}(N,\rho ,\eta )$ is a non-zero
rectifiable 1-current on $T^\ast=\underline{\underline{\tau }}(W,\theta ,\xi )$.
One may assume that $N=W$ by extending $\rho (x)=0$ and $\eta (x)=\xi (x)$
for $x\in W\setminus N$. Since $T^{\ast }$ is a solution to the ROTPB($\mu
,\nu $) problem and $\partial (T^\ast \pm S)\preceq \nu -\mu $, the function $%
g(t):=\mathbf{E}(T^{\ast }+tS)$ defined on the interval $[-1,1]$ achieves
its minimum value at $t=0$. Nevertheless,
\begin{eqnarray*}
g(t) &=&\mathbf{E}(T^{\ast }+tS)=\mathbf{M}_{\alpha }(T^{\ast
}+tS)-\int_{X}hd(\partial (T^{\ast }+tS)) \\
&=&\mathbf{E}(T^{\ast })+\int_{W}\left\vert \theta (x)+t\rho (x)\langle \xi
(x),\eta (x)\rangle \right\vert ^{\alpha }-\theta (x)^{\alpha
}d\mathcal{H}^{1}(x)-t\int_{X}hd(\partial S).
\end{eqnarray*}%
Here, the value of the inner product $\langle \xi (x),\eta (x)\rangle =\pm 1$ for $\mathcal{H}^1-a.e. x\in W$. Then,
\begin{equation*}
g^{\prime \prime }(0)=\alpha (\alpha -1)\int_{W}\theta (x)^{\alpha -2}\rho
(x)^{2}d\mathcal{H}^{1}(x)<0,
\end{equation*}%
since $0<\alpha<1$ and $S$ is non-zero. This says that $g$ can not achieve a local minimum at $t=0$, a contradiction.
\end{proof}

\subsection{Finite atomic case}
In the context of \changetext{finite} atomic
measures, Theorem \ref{thm: S is zero} has important implications for the structure of the
optimal transport path $T^*$ as demonstrated by the following results.

\begin{proposition}
\label{prop: set P}Suppose both
\begin{equation*}
\mu =\sum_{i=1}^{\ell}a_{i}\delta _{x_{i}}\text{ and }\nu
=\sum_{j=1}^{n}b_{j}\delta _{y_{j}}
\end{equation*}%
are two \changetext{finite} atomic measures on $X$, $0<\alpha <1$,
and $T^*\in Path(\mu^*, \nu^*)$ is a solution to the ROTPB($\mu ,\nu $) problem. Also, let
\[
P:=spt(\mu-\mu^\ast)\cup spt(\nu-\nu^\ast)\]
denote the union of the supports of the measures $\mu-\mu^*$ and $\nu-\nu^*$. Then each connected component of the support of $T^{\ast }$ contains at most one element of $P$.
\end{proposition}

\begin{proof}
Without loss of generality, we may assume that the support of $T^{\ast }$ is
connected, and we want to show that the set
\begin{eqnarray}
P &=&\{x_{i}:\mu ^{\ast }(\{x_{i}\})<\mu (\{x_{i}\})\}\bigcup \{y_{j}:\nu
^{\ast }(\{y_{j}\})<\nu (\{y_{j}\})\}  \label{P_set} \\
&=&\{p\in \{x_{1},\cdots ,x_{\ell},y_{1},\cdots ,y_{n}\}:(\nu -\nu ^{\ast
})\{p\}+(\mu -\mu ^{\ast })\{p\}>0\}  \notag
\end{eqnarray}
contains at most one element. Assume that $P$ has at least two distinct elements $p_{1}$ and $%
p_{2}$. Also, we may assume that $(\mu -\mu ^{\ast })\{p_{1}\}>0$ and $(\mu
-\mu ^{\ast })\{p_{2}\}>0$ (the proofs for the other cases are similar). \changetext{Since $T^*$ is acyclic (see \cite[Propositions 2.1 and 2.2]{xia2015motivations}), there exists a unique oriented curve $\gamma$} on the support of $T^{\ast }$ from $%
p_{1}$ to $p_{2}$, and set $S=\sigma\changetext{ I_\gamma}$ with
\[\sigma
=\min (\{\theta (x):x\in \gamma \},(\mu -\mu ^{\ast })\{p_{1}\},(\mu -\mu
^{\ast })\{p_{2}\})>0,\]
\changetext{and $I_\gamma$ being the rectifiable 1-current associated with $\gamma$ (see (\ref{eqn: I_gamma}) for the precise definition).} Then, $S$ is non-zero and on $T$ in the sense of Definition \ref{def: S_on_T}. Moreover, by the choice of $\sigma $,
\begin{equation*}
\changetext{\mu ^{\ast } \pm \sigma (\delta _{p_{2}}-\delta _{p_{1}})\leq \mu}.
\end{equation*}%
Thus,
\begin{equation*}
\partial (T\pm S)=\nu ^{\ast }-\mu ^{\ast }\pm \sigma (\delta
_{p_{2}}-\delta _{p_{1}})\preceq \nu -\mu .
\end{equation*}%
According to Theorem \ref{thm: S is zero}, $S$ must be zero, a
contradiction. 
\end{proof}

The set $P$ in Proposition \ref{prop: set P} represents the collection of boundary nodes on
which the amount of mass involved in the optimal transport path $T^{\ast }$
is smaller than its counterpart specified initially. The proof hinges on the
fact that if a connected component of the support of $T^{\ast }$ contains
two elements in $P$, one would be able to cut cost by reallocating the mass
transported along $T^{\ast }$, which however is precluded by Theorem \ref%
{thm: S is zero}.

According to Proposition \ref{prop: no_sigma_proportion}, in the \changetext{finite} atomic case, there exists at least one point $p$ on the support of $\mu^*$ or one point $q$ on the support of $\nu^*$, such that either
\begin{equation}
\label{eqn: agree}
    \mu^*(\{p\})=\mu(\{p\}) \text{ or } \nu^*(\{q\})=\nu(\{q\}).
\end{equation}
Proposition \ref{prop: set P} says that with at most one exception on each connected component, equation (\ref{eqn: agree}) holds for all points $p$ or $q$ on the supports of $\mu^*$ or $\nu^*$, respectively. Consequently,
with the help of Proposition \ref{prop: set P} and Corollary \ref{cor: zero_energy},
we may prove Theorem \ref{thm: atomic} as follows.

\textit{Proof of Theorem \ref{thm: atomic}.}
By Proposition \ref{prop: set P}, each $K_k$ contains at most one element of the set $P$. Thus, one of the following two cases holds:
\begin{itemize}
\item[Case 1:] \[
\mu ^{\ast } \mres K_k=\mu \mres K_k-m_k\delta_{p_k} \text{ and } \nu ^{\ast } \mres K_k=\nu \mres K_k
\]
for some point $p_k\in K_k \cap spt(\mu^\ast)$ and some real number $m_k \ge 0$.
\item[Case 2:] \[
\mu ^{\ast } \mres K_k=\mu \mres K_k \text{ and } \nu ^{\ast } \mres K_k=\nu \mres K_k-n_k \delta_{q_k}
\]
for some point $q_k\in K_k \cap spt(\nu^\ast)$ and some real number $n_k \ge 0$.
\end{itemize}

In the first case,
\[\mu^\ast(K_k)=\mu(K_k)-m_k \text{ and } \nu^\ast(K_k)=\nu(K_k). \]
Since $\mu(K_k)\ge\mu^\ast(K_k)=\nu^\ast(K_k)=\nu(K_k)$, it follows that
\[m_k=\mu(K_k)-\mu^\ast(K_k)=\mu(K_k)-\nu(K_k)=\max\{\mu(K_k)-\nu(K_k), 0\}.\]

Analogously, in the second case, we pick
\[
n_k=\max\{\nu(K_k)-\mu(K_k), 0\}
\]
as desired.
\qed

If the measure of mass at each source node is sufficiently large, all source
nodes would fall into the set $P$, yielding a natural partition of the
transport path $T^{\ast}$ as stated in the following corollary. In this
case, destination nodes can be classified by the source node from which they
receive the mass. Under a symmetric condition, a similar decomposition exists for destination nodes.

\begin{corollary}
\label{corollary: single_source} Suppose both
\begin{equation*}
\mu =\sum_{i=1}^\ell a_i \delta_{x_i} \text{ and }\nu=\sum_{j=1}^n b_j
\delta_{y_j}
\end{equation*}
are (positive) \changetext{finite} atomic measures on $X$, and $T^*$ is a solution to the ROTPB($%
\mu, \nu$) problem.
\begin{itemize}
    \item[(a)] If
\begin{equation}
\min_{1\le i\le \ell}{a_i} \ge \sum_{j=1}^n b_j,  \label{eqn: allocation}
\end{equation}
then $T^*$ can be decomposed as $T^*=T_1+T_2+\cdots+T_\ell$, where for each $%
i=1,\cdots, \ell$, $T_i$ is an $\alpha$-optimal transport path from a single
source located at $x_i$.
\item[(b)] Similarly, if
\begin{equation}
\min_{1\le j\le n}{b_j} \ge \sum_{i=1}^\ell a_i,  \label{eqn: allocation_b}
\end{equation}
then $T^*$ can be decomposed as $T^*=T_1+T_2+\cdots+T_n$, where for each $%
j=1,\cdots, n$, $T_j$ is an $\alpha$-optimal transport path to a single
destination located at $y_j$.
\end{itemize}

\end{corollary}

\begin{proof}
We only need to prove case $(a)$ as $(b)$ follows from a symmetric argument. To do so, it is sufficient to show that each connected component of the support of $%
T^{\ast}$ contains only one source point in $\{x_{1},x_{2},\cdots ,x_{\ell}\}$%
. We prove it by contradiction. Assume that there exists a connected
component of the support of $T^{\ast }$ that contains at least two sources,
say $x_{1}$ and $x_{2}$. Then
\begin{equation*}
\mu ^{\ast }(\{x_{1}\})>0\text{ and }\mu ^{\ast }(\{x_{2}\})>0.
\end{equation*}%
As a result,
\begin{equation*}
\mu ^{\ast }(\{x_{1}\})<\mu ^{\ast }(\{x_{1}\})+\mu ^{\ast }(\{x_{2}\})\leq
||\mu ^{\ast }||=||\nu ^{\ast }||\leq \sum_{j=1}^{n}b_{j}\leq a_{1}=\mu(\{x_{1}\}),
\end{equation*}%
by (\ref{eqn: allocation}). This shows that $x_{1}$ belongs to the set $P$
in (\ref{P_set}). Similar argument leads to $x_{2}\in P$. This contradicts
Proposition \ref{prop: set P}. \changetext{Let $\{K_i: i=1,2,\cdots, \ell\}$ be the connected components of the support of $T^*$, and set $T_i=T\mres K_i$. Since $T$ is $\alpha$-optimal, and $\{K_i\}$ are pairwise disjoint, each $T_i$ is also $\alpha$-optimal.} 
\end{proof}

\subsection{General case}
In what follows, we generalize the results of Theorem \ref{thm: atomic} for $\mu $ and $\nu $ being any two Radon measures, not necessarily finite atomic. To do so, we adopt a Lagrangian approach, and follow some notations used in \cite{colombo2020}.

By Theorem \ref{thm: existence}, the ROTPB($\mu ,\nu $) problem has a solution
\begin{equation}
\label{eqn: T-star}
    T^*=\underline{\underline{\tau}}(W,\varphi,\zeta)\in Path(\mu^*, \nu^*).
\end{equation}

We denote by $\Gamma$ the space of 1-Lipschitz curves $\gamma: [0,\infty)\rightarrow \mathbb{R}^m $, which are eventually constant (and hence of finite length). For $\gamma\in \Gamma$, we denote the values
\[t_0(\gamma):=\sup\{t: \gamma \text{ is constant on }[0,t]\}\]
and
\[t_\infty(\gamma):=\inf\{t: \gamma \text{ is constant on }[t, \infty)\},\]
and denote $\gamma(\infty):=\lim_{t\rightarrow \infty}\gamma(t)$.
Given $\gamma\in \Gamma$, the projections of $\gamma$ onto its starting and stopping points are
\begin{eqnarray}
\label{eqn: projections}
p_{0}(\gamma):=\gamma(0) \text{ and } p_{\infty}(\gamma):=\gamma(\infty).
\end{eqnarray}

We say that a curve $\gamma\in \Gamma$ is \textit{simple} if $\gamma(s)\ne \gamma(t)$ for every $t_0(\gamma)\le s< t \le t_\infty(\gamma)$. \changetext{In particular, $\gamma$ is non-constant in any non-trivial sub-interval $[s,t] \subseteq [t_0(\gamma), t_\infty(\gamma)]$.}

For each simple curve $\gamma\in \Gamma$, we may canonically associate it with the rectifiable 1-current
\begin{equation}
    \label{eqn: I_gamma}
    I_\gamma:=\rect\left(\im(\gamma), \frac{\gamma'}{\abs{\gamma'}}, 1\right),
\end{equation}
where $\im(\gamma)$ denotes the image of the curve $\gamma$ in $\mathbb{R}^m$.
It is easy to check that $\mathbf{M}(I_\gamma)=\mathcal{H}^1(\im(\gamma))$ and $\partial I_\gamma=\delta_{\gamma(\infty)}-\delta_{\gamma(0)}$; since $\gamma$ is simple, if it is also non-constant, then $\gamma(\infty)\ne \gamma(0)$ and $\mathbf{M}(\partial I_\gamma)=2.$

A normal current $T\in \mathcal{D}_1(\mathbb{R}^m)$ is said \textit{acyclic} if there exists no non-trival current $S$ such that $\partial S=0$ and $\mathbf{M}(T)=\mathbf{M}(T-S)+\mathbf{M}(S)$.

Now, we recall a fundamental result of Smirnov in \cite{smirnov},
which establishes that every acyclic normal
1-current can be written as a weighted average of simple Lipschitz curves in the following sense.

\begin{definition}
Let $T$ be a normal 1-current in $\mathbb{R}^m$ represented as a vector-valued measure $\vec{T}\abs{T}$, and let
$\eta$ be a finite positive measure on $\Gamma$
such that
\begin{equation}
\label{eqn: T_eta}
    T=\int_\Gamma I_\gamma d\eta(\gamma)
\end{equation}
in the sense that for every smooth compactly supported 1-form $\omega\in \mathcal{D}^1(\mathbb{R}^m)$, it holds that
\begin{equation}
    T(\omega)=\int_\Gamma I_\gamma(\omega) d\eta(\gamma).
\end{equation}
We say that $\eta$ is a \textit{good decomposition} of $T$ if $\eta$ is supported on non-constant, simple curves and
satisfies the following equalities:
\begin{itemize}
    \item[(a)] $\mathbf{M}(T)=\int_\Gamma \mathbf{M}(I_\gamma) d\eta(\gamma)=\int_\Gamma \mathcal{H}^1(Im(\gamma)) d\eta(\gamma)$;
    \item[(b)] $\mathbf{M}(\partial T)=\int_\Gamma \mathbf{M}(\partial I_\gamma) d\eta(\gamma)=2 \eta(\Gamma)$.
\end{itemize}
\end{definition}
It has been shown in \cite[Theorem 10.1]{paolini} that optimal transport paths $T^*$ \changetext{ with finite $\mathbf{M}_\alpha$ cost} are acyclic,
and hence they admit such a good decomposition.

In the next result, we collect some useful properties of good decompositions, whose proof can be found in \cite[Proposition 3.6]{colombo}.

\begin{theorem}(Existence and properties of good decompositions)\cite[Theorem 5.1]{paolini2012} and \cite[Proposition 3.6]{colombo}. Let $T$ be an $\alpha$-optimal transport path from $\mu^-$ to $\mu^+$ with finite $\mathbf{M}_\alpha$ cost. Then $T$ is acyclic and there is a Borel finite measure $\eta$ on $\Gamma$ such that $\eta$ is a good decomposition of $T$. Moreover, if $\eta$ is a good decomposition of $T$, the following statements hold:
\begin{itemize}
    \item $\mu^-=\int_\Gamma \delta_{\gamma(0)}d\eta(\gamma),\mu^+=\int_\Gamma \delta_{\gamma(\infty)}d\eta(\gamma) $.
    \item If $T=\rect\left(M, \theta, \xi\right)$ is rectifiable, then
    \begin{equation}
    \label{eqn: theta}
        \theta(x)=\eta(\{\gamma\in \Gamma: x\in Im(\gamma)\})
    \end{equation}
    for $\mathcal{H}^1$-a.e. $x\in M$.
    \item For every $\tilde{\eta}\changetext{\leq} \eta$, the representation
    \[\tilde{T}=\int_\Gamma I_\gamma d\tilde{\eta}(\gamma)\]
    is a good decomposition of $\tilde{T}$. Moreover, if $T=\rect\left(M, \theta, \xi\right)$ is rectifiable, then $\tilde{T}$ can be written as $\tilde{T}=\rect\left(M, \tilde{\theta}, \xi\right)$ with $\tilde{\theta}(x)\le \min\{\theta(x), \tilde{\eta}(\Gamma)\} $ for $\mathcal{H}^1$-a.e. $x\in M$.
\end{itemize}
\end{theorem}

We now introduce the following notations.
\changetext{
Let $T^*=\underline{\underline{\tau}}(W,\varphi,\zeta)\in Path(\mu^*, \nu^*)$ be given as in (\ref{eqn: T-star}), and $\eta$ be a good decomposition of $T^*$. Denote
\[\tilde{\mu}=\mu-\mu^* \text{ and } \tilde{\nu}=\nu-\nu^*.\]
For any $x\in W$, let us denote by  $\Gamma(x)$ the set of simple curves $\gamma\in \Gamma$ such that $x\in Im(\gamma)$. By equation (\ref{eqn: theta}), $\varphi(x)=\eta(\Gamma(x))$ for $\mathcal{H}^1$-a.e. $x\in W$.

\begin{proposition}
\label{prop: main_prop}
For any $x\in W$ with $\varphi(x)=\eta(\Gamma(x))>0$, denote
\[\bar{\mu}_x=(p_0)_{\#}\left(\eta \mres{\Gamma(x)}\right) \text{ and } \bar{\nu}_x=(p_\infty)_{\#}\left(\eta \mres{\Gamma(x)}\right)\]
where $p_0$ and $p_\infty$ are projections given in (\ref{eqn: projections}).
Let 
\[\tilde{\mu}=\tilde{\mu}_x^{ac}+\tilde{\mu}_x^s, \; \tilde{\mu}_x^{ac} \ll \bar{\mu}_x, \; \tilde{\mu}_x^{s} \perp \bar{\mu}_x \]
be the Lebesgue-Radon-Nikod\'ym decomposition of $\tilde{\mu}$ with respect to $\bar{\mu}_x$, and
\[\tilde{\nu}=\tilde{\nu}_x^{ac}+\tilde{\nu}_x^s, \; \tilde{\nu}_x^{ac} \ll \bar{\nu}_x, \; \tilde{\nu}_x^{s} \perp \bar{\nu}_x\]
be the Lebesgue-Radon-Nikod\'ym decomposition of
$\tilde{\nu}$ with respect to $\bar{\nu}_x$.
Then
\begin{enumerate}
    \item[(a)] There exist $p_x\in W $, $q_x\in W $, $m_x\ge 0$, and $n_x\ge 0$
such that
\[\tilde{\mu}_x^{ac}=m_x \delta_{p_x} \text{ and } \tilde{\nu}_x^{ac}=n_x \delta_{q_x}.\]
\item[(b)] If $m_x>0$, then $\bar{\mu}_x(\{p_x\})>0$ and there exists a Lipschitz curve $\gamma^-_x$ from $p_x$ to $x$ such that $\varphi(y)\ge \bar{\mu}_x(p_x)$ for $\mathcal{H}^1$-a.e. $y\in Im(\gamma^-_x)$.
\item[(c)] If $n_x>0$, then $\bar{\nu}_x(\{q_x\})>0$ and there exists a Lipschitz curve $\gamma^+_x$ from $x$ to $q_x$ such that $\varphi(y)\ge \bar{\nu}_x(q_x)$ for $\mathcal{H}^1$-a.e. $y\in Im(\gamma^+_x)$.
\item[(d)] At least one of $m_x$ and $n_x$ is zero.
\end{enumerate}
\end{proposition}

\begin{proof}
Note that
\[\bar{\mu}_x(X)=(p_0)_{\#}\left(\eta \mres {\Gamma(x)}\right)(X)=(\eta \mres {\Gamma(x)})(p_0^{-1}(X))=\eta(\Gamma(x))=\varphi(x)>0,\]
and $\bar{\mu}_x\leq (p_0)_{\#}\eta=\mu^*$.
For the sake of contradiction, we assume that $\tilde{\mu}_x^{ac}$ is non-zero and not a multiple of a Dirac mass. Then there exists a Borel measurable
set $R$ such that
\begin{equation*}
\tilde{\mu}_x^{ac}(R)>0 \text{ and } \tilde{\mu}_x^{ac}(X\setminus R)>0.
\end{equation*}%
Since $\tilde{\mu}_x^{ac} \ll \bar{\mu}_x$, there exists a non-negative function $g\in \mathcal{L}^1(X,\bar{\mu}_x)$ such that $\tilde{\mu}_x^{ac}= g\cdot\bar{\mu}_x$. We define
\[\mu_0:=\min\{g(\cdot), 1\} \cdot\bar{\mu}_x. \]
Still we have $\mu_0(R)>0$ and $\mu_0(X\setminus R)>0$. Without loss of generality, we may assume that $0<\mu_0(R)\le \mu_0(X\setminus R)$.
Setting \[\mu_1=\mu_0\mres R \text{ and } \mu_2=\frac{\mu_0(R)}{\mu_0(X\setminus R)}\left(\mu_0\mres (X\setminus R)\right)\]
yields two positive measures $\mu _{1}$ and $\mu _{2}$ such that
$\mu _{1}$ is concentrated on $R$ and $\mu _{2}$ is concentrated on $%
X\setminus R$ with equal mass.
Moreover, since both $\mu_1$ and $\mu_2$ are absolutely continuous with respect to $\bar{\mu}_x$, there exist two non-negative $\bar{\mu}_x$-integrable functions $\rho_1 $ and $\rho_2$ such that
\[\mu_1=\rho_1 \bar{\mu}_x \text{ and } \mu_2=\rho_2 \bar{\mu}_x.\]
Let $\rho=\rho_1-\rho_2$. Note that $\abs{\rho}\le 1$. From the construction of $\mu_1$ and $\mu_2$, we have
\[\bar{\mu}_x(\rho)=\int_X \rho d \bar{\mu}_x=\int_X \rho_1 d \bar{\mu}_x-\int_X \rho_2 d \bar{\mu}_x=0,\]
and 
\[\int_X \abs{\rho} d \bar{\mu}_x=2\int_X \rho_1 d \bar{\mu}_x=2 \tilde{\mu}_x^{ac}(R)>0.\]

Now, for any $\gamma\in \Gamma(x)$, let $\gamma^-$ be the part of $\gamma$ from $\gamma(0)$ to $x$.
Define
    \[S:=\int_{\Gamma(x)} \rho(\gamma(0))I_{\gamma^-} d\eta(\gamma).\]
Then for any smooth function $f$ with a compact support, we have
    \begin{eqnarray*}
    \partial S(f)&=&S(df)\\
        &=&\int_{\Gamma(x)} \rho(\gamma(0))I_{\gamma^-}(df) d\eta(\gamma)\\
    &=&\int_{\Gamma(x)} \rho(\gamma(0))\partial(I_{\gamma^-})(f) d\eta(\gamma)\\
    &=&\int_{\Gamma(x)} \rho(\gamma(0))(f(x)-f(\gamma(0)) d\eta(\gamma)\\
    &=&\int_{\Gamma(x)} \rho(\gamma(0))f(x) d\eta(\gamma)-\int_{\Gamma(x)} \rho(\gamma(0))f(\gamma(0)) d\eta(\gamma)\\
    &=&f(x)\bar{\mu}_x(\rho)-(\rho\bar{\mu}_x)(f)=-(\rho\bar{\mu}_x)(f).
    \end{eqnarray*}
    Therefore, $\partial S=-\rho \bar{\mu}_x \ne 0$.
    
Since $T^*$ is rectifiable, by construction $S$ is also rectifiable. We may write it as $S=\underline{\underline{\tau}}(M_S,\theta_S,\xi_S)$ for some $M_S\subseteq W$. At $\mathcal{H}^1$-a.e. $y\in M_S$,
\[\theta_S(y)\le \int_{\Gamma(x)\cap \Gamma(y)} \abs{\rho(\gamma(0))}d\eta(\gamma)\le \int_{\Gamma(x)\cap \Gamma(y)} 1 d\eta(\gamma)\le \int_{\Gamma(y)} 1 d\eta(\gamma)=\varphi(y).\]
This shows that $S$ is on $T^*$ in the sense of Definition \ref{def: S_on_T}. 

We now show that $\partial(T^* \pm S)\leq \nu-\mu$. Given
\[\partial(T^* \pm S)=\nu^*-\mu^* \pm (\rho \bar{\mu}_x)=\nu^*-(\mu^* \mp (\rho \bar{\mu}_x)),\] it is sufficient to show that $\mu^* \mp (\rho \bar{\mu}_x)\leq \mu$ as (positive) measures, which is the case provided that $\bar{\mu }_x \leq \mu ^*$ and $\abs{\rho}\le 1$. Also,
\[\mu-\left(\mu^* \mp (\rho \bar{\mu}_x)\right)=\mu-\mu^* \pm (\rho \bar{\mu}_x)=\tilde{\mu}\pm (\mu_1-\mu_2)\] are positive measures because $\mu_1\leq \bar{\mu}_x^{ac}\leq \tilde{\mu}$ and similarly $\mu_2\leq \tilde{\mu}$. As a result, $\partial(T^* \pm S)\preceq \nu-\mu$. By Theorem \ref{thm: S is zero}, $S$ is zero which contradicts $\partial S\neq 0$. Therefore, $\bar{\mu}_x^{ac}$ must be in the form of $m_x \delta_{p_x}$ for some $m_x\ge 0$ and $p_x\in W$. Similarly, we have $\bar{\nu}_x^{ac}=n_x\delta_{q_x}$ for some $n_x\ge 0$ and $q_x\in W$. This proves part (a).

Now assume that $m_x>0$. Since $\tilde{\mu}_x^{ac}=m_x \delta_{p_x}\ll \bar{\mu}_x$, we have $\bar{\mu}_x(\{p_x\})>0$. That is,
\[0<(p_0)_{\#}\left(\eta \lfloor_{\Gamma(x)}\right)(\{p_x\})= \eta (\{\gamma\in \Gamma(x): \gamma(0)=p_x\}).\]
Because $T^*$ is acyclic and $\eta$ is a good decomposition of $T^*$, for $\eta$-a.e. $\gamma\in \Gamma(x)$ with $\gamma(0)=p_x$, the image $Im(\gamma)$ of $\gamma$ shares a common Lipschitz curve $\gamma^-_x$ in $W$ from $p_x$ to $x$. For $\mathcal{H}^1$-a.e. $y$ on $Im(\gamma^-_x)$,
\[\varphi(y)= \eta (\{\gamma\in \Gamma(y)\})\ge \eta (\{\gamma\in \Gamma(x): \gamma(0)=p_x\})=\bar{\mu}_x(\{p_x\}). \]
This proves part (b). Similar arguments lead to part (c).

Suppose by contradiction that both $m_x>0$ and $n_x>0$. Then, by parts (b) and (c), we consider the rectifiable 1-current
\[S_x:=\sigma \left(I_{\gamma^-_x}+I_{\gamma^+_x}\right),\]
for $\sigma=\min\{\bar{\mu}_x(\{p_x\}), \bar{\nu}_x(\{q_x\}), m_x, n_x\}>0$. Clearly, $S_x$ is on $T^*$, non-zero, and $\partial(T^* \pm S_x) \preceq \nu-\mu$. This contradicts Theorem \ref{thm: S is zero}. Therefore, at least one of $m_x$ and $n_x$ is zero.
\end{proof}

To derive the generalized version of Theorem \ref{thm: atomic}, we introduce the concept of path-connectivity on rectifiable 1-currents as follows.

\begin{definition}
\label{def: path-connected} Let $T=\underline{\underline{\tau}}(M,\theta,\xi)$ be a rectifiable 1-current. For any two points $x_1, x_2 \in X$,
we say $x_1$ and $x_2$ are $T$-path-connected if there exists a Lipschitz curve $\gamma:[0,1]\rightarrow X$ such that $\gamma(0)=x_1, \gamma(1)=x_2$, 
$\mathcal{H}^1(Im(\gamma)\setminus M)=0$, and there exists a number $c>0$ such that $\theta(z)\ge c$ for $\mathcal{H}^1$-a.e. $z\in Im(\gamma)\cap M$. 
\end{definition}
The $T$-path-connectivity defines an equivalence relation on $X$.  A rectifiable 1-current $T=\underline{\underline{\tau}}(M,\theta,\xi)$ is called path-connected if every two points on $M$ are $T$-path-connected.

For any $T$-path-connected component $M'$ of $M$, we consider the restriction $T\mres M'$ of $T$ on $M'$. $T\mres M'$ is zero if $M'$ contains only one point. In this case, we say that the component $M'$ is degenerate.   When $M'$ is non-degenerate, i.e., it contains at least two distinct points $x_1$ and $x_2$, we have
\[\mathbf{M}(T\mres M')=\int_{M'}\theta d\mathcal{H}^1\ge c \mathcal{H}^1(Im(\gamma))>0\] 
using the notations given in Definition \ref{def: path-connected}. Since $\mathbf{M}(T)<\infty$, $M$ has at most countably many non-degenerate $T$-path-connected components.

Observe that non-degenerate components may fail to exist even if $\mathbf{M}(T)>0$. For instance, let $C$ (e.g., a fat-Cantor set) be a nowhere dense subset of $[0,1]$ with $0<\mathcal{H}^1(C)<1$. Then, for $S=\underline{\underline{\tau}}(C,\chi_C,1)$, each $S$-path-connected component is degenerate. Luckily, the following lemma indicates that each non-zero $\alpha$-optimal transport path has at least one non-degenerate path-connected component.

\begin{lemma}
\label{lemma: path_connected}
 Let $T=\underline{\underline{\tau}}(M,\theta,\xi)$ be a non-zero $\alpha$-optimal transport path for some $0<\alpha<1$. Then \[T=\sum_{i\in J} T\mres M_i,\] where $\{M_i: i\in J\}$ are the collection of all non-degenerate $T$-path-connected components of $M$, and $J$ is a non-empty countable set.
\end{lemma}

To prove Lemma \ref{lemma: path_connected}, we first recall the notation of superlevel set as introduced in \cite{xia_boundary}: For any $\lambda>0$, the $\lambda$-superlevel set of a rectifiable current $T=\underline{\underline{\tau}}(M,\theta,\xi)$ is the set
\[M_{\lambda}:=\{p\in M: \theta(p)\ge \lambda\}.\]

\begin{lemma}(\cite[Proposition 4.3]{xia_boundary})
 Let $T=\underline{\underline{\tau}}(M,\theta,\xi)$ be any $\alpha$-optimal transport path. Then for any $\sigma_1>\sigma_2>0$ and any $p\in M_{\sigma_1}$, there exists an open ball neighborhood $B_r(p)$ of $p$ such that 
\begin{equation}
\label{eqn: superlevel}
   M_{\sigma_1}\cap B_r(p)\subseteq \text{ the support of } Q_p\subseteq M_{\sigma_2}\cap B_r(p), 
\end{equation}
where $Q_p=\sum_{i=1}^K m_i\Gamma_i$ is a bi-Lipschitz chain.
\end{lemma}
Here, as stated in \cite[Corollary 4.2]{xia_boundary}, each $\Gamma_i$ is a bi-Lipschitz curve from $p$. These bi-Lipschitz curves $\Gamma_i$ are pairwise disjoint except at their common endpoint $p$, and $K$ is a universal constant.

\begin{proof}[Proof of Lemma \ref{lemma: path_connected}:]
Let $\{M_i: i\in J\}$ be the collection of all non-degenerate $T$-path-connected components of $M$, where $J$ is countable.  For any $p\in M$ with $\theta(p)>0$, let $\sigma_1=\theta_p$ and $\sigma_2=\frac{1}{2}\sigma_1$. By (\ref{eqn: superlevel}), any point on the support of the bi-Lipschitz curve $Q_p$ is $T$-path-connected with $p$. Hence, $p$ belongs to some non-degenerate $T$-path-connected component $M_i$ for some $i\in J$.
As a result, we decompose 
$M_+=\{x\in M: \theta(x)>0\}$ as the disjoint union of $M_i$ with $i\in J$. Thus, $T=\sum_{i\in J} T\mres M_i$.
\end{proof}

We now go back to the study of $T^*$, which is also an $\alpha$-optimal transport path. Consequently, one can write
 \[T^*=\sum_{i\in J} T^*\mres W_i,\] where $\{W_i: i\in J\}$ are the collection of all non-degenerate $T^*$-path-connected components of $W$, and $J$ is a non-empty countable set if $T^*$ is non-zero.
 
\begin{lemma}
\label{lemma: same_component}
Suppose $x_1$ and $x_2$ belong to the same  non-degenerate $T^*$-path-connected component  of $W$. Then, at most one of $\{m_{x_1}, n_{x_1}, m_{x_2}, n_{x_2}\}$ is non-zero.
\end{lemma}

\begin{proof}
Otherwise, let us just assume $m_{x_1}>0$ and $ n_{x_2}>0$, with the proof for other cases following similarly. Since  $x_1$ and $x_2$ are path-connected on $T^*$, there exists a Lipschitz curve $\gamma_{x_1}^{x_2}$ from $x_1$ to $x_2$ such that
$\mathcal{H}^1(Im(\gamma_{x_1}^{x_2})\setminus W)=0$, and there exists a number $c>0$ such that $\varphi(z)\ge c$ for $\mathcal{H}^1$-a.e. $z\in Im(\gamma_{x_1}^{x_2})$. Now, we consider the rectifiable 1-current
\[S:=\sigma \left(I_{\gamma^-_{x_1}}+I_{\gamma_{x_1}^{x_2}}+I_{\gamma^+_{x_2}}\right),\]
for $\sigma=\min\{\bar{\mu}_{x_1}(\{p_{x_1}\}), \bar{\nu}_{x_2}(\{q_{x_2}\}), m_{x_1}, n_{x_2}, c\}>0$. Clearly, $S$ is on $T^*$, non-zero, and $\partial(T^* \pm S_x) \preceq \nu-\mu$. This contradicts Theorem \ref{thm: S is zero}. 
\end{proof}

For each $i\in J$, if there exists one point $x\in W_i$ such that $m_x>0$, then by part(b) of Proposition \ref{prop: main_prop}, the associated point $p_x$ is also $T^*$-path-connected with $x$ and hence $p_x\in W_i$. By Lemma \ref{lemma: same_component}, $m_x\delta_{p_x}$ is independent of the choice of $x\in W_i$, and thus can be represented by $m_i\delta_{p_i}$. If $m_x=0$ for all $x\in W_i$, we simply pick $p_i$ to be any fixed point in $W_i$ and set $m_i=0$. Analogously, we denote $n_x\delta_{q_x}$ by $n_i\delta_{q_i}$ for each $i\in J$. As a result, we arrive at two atomic measures
\begin{equation}
\label{eqn: measures_ab}
    \mathbf{a}=\sum_{i\in J}m_i\delta_{p_i} \text{ and } \mathbf{b}=\sum_{i\in J}n_i\delta_{q_i},
\end{equation}
where either $m_i=0$ or $n_i=0$ for each $i\in J$.

\begin{lemma} 
\label{lemma: perp}
It holds that
\[\left(\tilde{\mu}-\mathbf{a}\right) \perp \mu^*, \mathbf{a}\ll \mu^* \text{ and } \left(\tilde{\nu}-\mathbf{b}\right) \perp \nu^*, \mathbf{b}\ll \nu^*.\]
\end{lemma}
\begin{proof}
Let $\hat{\mu}=\tilde{\mu}-\mathbf{a}$. Then for each $x\in W$ with $\varphi(x)>0$, by Proposition \ref{prop: main_prop},
\[\hat{\mu}\perp \bar{\mu}_x. \]
Thus, there exists a $\hat{\mu}$-negligible set $A_x$ such that $\bar{\mu}_x(A_x)=\bar{\mu}_x(X)$.
Observe that one may pick countably many points $\{x_k: \varphi(x_k)>0\}_{k=1}^\infty \subset W$ so that for $\eta$-a.e. $\gamma\in \Gamma$, $\gamma$ passes at least one of these points. One way to select these points is taking a countable dense subset of the 1-rectifiable set $W_+:=\{x\in W: \varphi(x)>0\}$. Now for each $k$,
\[\eta\left(\Gamma(x_k) \right)=\bar{\mu}_{x_k}(X)=\bar{\mu}_{x_k}(A_{x_k})=\eta\left(\{\gamma: \gamma(0)\in A_k\} \right),\]
and thus
\[\eta(\Gamma)=\eta\left(\bigcup_k \Gamma(x_k) \right)=\eta\left(\{\gamma\in \Gamma: \gamma(0)\in \bigcup_k A_{x_k}\}\right).\]
As a result,
\[\mu^*(X)=\mu^*(\bigcup_k A_{x_k}) \text{ and } \hat{\mu}(\bigcup_k A_{x_k})=\sum_k \hat{\mu}(A_{x_k})=0.\]
Therefore, $\mu^*\perp \hat{\mu}$ as desired. For each $i\in J$, assume $m_i\delta_{p_i}=m_x\delta_{p_x}$, then 
\[m_i\delta_{p_i}\ll \bar{\mu}_x\ll \mu^*.\]
Thus, $\mathbf{a}\ll \mu^*$. Similarly, we have $\left(\tilde{\nu}-\mathbf{b}\right) \perp \nu^*$ and $\mathbf{b}\ll \nu^*$.
\end{proof}

\begin{lemma}
\label{lemma: restrict}
For any two (positive) measures $\mu_1$ and $\mu_2$. Let $\lambda=\mu_1+\mu_2$. If $\mu_1\perp \mu_2$, then there exists a $\lambda$-measurable set $A$ such that $\mu_1=\lambda\mres A$ and $\mu_2=\lambda \mres (X\setminus A)$. 
\end{lemma}
\begin{proof}
Since $\lambda=\mu_1+\mu_2$, $\mu_1\ll \lambda$ and $\mu_2\ll \lambda$. By the Radon-Nikod\'ym theorem, there exists a non-negative $\lambda$-measurable function $f$ such that $\mu_1=f\lambda$ and $\mu_2=(1-f)\lambda$. Given $\mu_2$ is a positive measure, it follows that $0\le f(x)\le 1$ for $\lambda$-a.e. $x$.  Let $K=\{x: 0<f(x)<1\}$. We claim that $\lambda(K)=0$. Indeed, assume $\lambda(K)>0$. Then, $\mu_1(K)=\int_Kf(x)d\lambda(x)>0$ and similarly  $\mu_2(K)=\int_K(1-f(x))d\lambda(x)>0$. This contradicts $\mu_1\perp \mu_2$. Therefore, $\lambda(K)=0$.
Setting $A=\{x: f(x)=1\}$ yields $\mu_1=\lambda\mres A$ and $\mu_2=\lambda \mres (X\setminus A)$. 
\end{proof}
Combining the preceding results leads to the following theorem, which is a generalized version of Theorem \ref{thm: atomic}.
\begin{theorem}
\label{main} Suppose $\mu $ and $\nu $ are two Radon measures on
$X$, $0<\alpha <1$, and $T^{\ast }=\underline{\underline{\tau}}(W,\varphi,\zeta)\in Path(\mu ^{\ast },\nu ^{\ast })$ is a
solution to the ROTPB($\mu ,\nu $) problem. Then
\begin{enumerate}
    \item There exist two atomic measures $\mathbf{a}$ and $\mathbf{b}$, a $\mu$-measurable set $A$, and a $\nu$-measurable set $B$ such that
\[\mu\mres A=\mu^*+\mathbf{a} \text{ and } \nu\mres B=\nu^*+\mathbf{b}. \]
\item Let $\{W_i: i\in J\}$ be the collection of all non-degenerate $T^*$-path-connected components of $W$. Then, for each $i\in J$, there exist $m_i\ge 0 \text{ and } n_i\ge 0$ with either $m_i=0$ or $n_i=0$; and two points $ p_i, q_i\in W_i$ such that 
\[\mathbf{a}=\sum_{i\in J}m_i\delta_{p_i} \text{ and } \mathbf{b}=\sum_{i\in J}n_i\delta_{q_i}. \]
\end{enumerate}
\end{theorem}

The theorem says that on each path-connected component, at most one atom is not fully in use. In particular, when both $\mu$ and $\nu$ are atom-free\footnote{ \changetext{A measure $\mu$ is called atom-free if $\mu(\{p\})=0$ for every $p\in X$.}} measures, it follows that
\begin{equation*}
\mu ^{\ast }=\mu \mres{A}\text{ and } \nu
^{\ast }=\nu \mres{B}.
\end{equation*}
\begin{proof}
The atomic measures $\mathbf{a}$ and $\mathbf{b}$ are obtained by (\ref{eqn: measures_ab}). By Lemma \ref{lemma: perp}, $(\mu-\mathbf{a}-\mu^*)\perp \mu^*$. By Lemma \ref{lemma: restrict}, there exists a $(\mu-\mathbf{a})$-measurable set $A$ such that $\mu^*=(\mu-\mathbf{a})\mres A$. Since $\mu^*$ concentrates on $A$ and $\mathbf{a}\ll \mu^*$, we have $\mathbf{a}\mres A=\mathbf{a}$ and $A$ is also $\mu$-measurable. Thus, $\mu\mres A=\mu^*+\mathbf{a}$. Similarly, we have $\nu\mres B=\nu^*+\mathbf{b}$ for some $\nu$-measurable set $B$.
\end{proof}
}
In light of the theorem, on locations involving mass transportation, the
measure $\mathbf{a}$, which represents the mass left unmoved by the solution
$T^{\ast }$, must be atomic, so is measure $\mathbf{b}$ which
summarizes the distribution of excess demand at destinations. This is
because if not the planner can exploit further gains by relocating the mass
moved along the path $T^{\ast }$ due to the efficiency in group
transportation.

\section{The impact of boundary payoff}\label{sec: impact of h}

An important deviation of the ROTPB problem from the literature is the
dependence of its solution on the boundary payoff as exemplified by Proposition %
\ref{prop: Tzero}. To gain further insights, in what follows we examine the
implications of the payoff function $h$ for the problem. For the sake of
expositional tractability, we assume that $\mu$ and $\nu$ are \changetext{disjointly supported} (i.e., $spt(\mu )\cap spt(\nu )=\emptyset $)
and the function $h$ takes the form
\begin{equation}
\label{eqn: h}
h(x)=%
\begin{cases}
c_{\mu },\text{ if }x\in spt(\mu ) \\
c_{\nu },\text{ if }x\in spt(\nu )%
\end{cases}%
\end{equation}%
where $c_{\mu }$ and $c_{\nu }$ are constants. In this case, for any $T\in
Path(\tilde{\mu},\tilde{\nu})$,
\begin{equation*}
\mathbf{E}_{\alpha}^h(T)=\mathbf{M}_{\alpha }(T)-\int_{X}c_{\nu }d\tilde{\nu}%
+\int_{X}c_{\mu }d\tilde{\mu}=\mathbf{M}_{\alpha }(T)-2c||\tilde{\mu}||=%
\mathbf{M}_{\alpha }(T)-c\mathbf{M}(\partial T),
\end{equation*}%
where $c=\frac{c_{\nu }-c_{\mu }}{2}$. The corresponding ROTPB($\mu ,\nu $) problem in this case becomes: Minimize
\begin{equation}
\label{eqn: E_c}
\mathbf{E}_{\alpha}^{c}(T):=\mathbf{M}_{\alpha }(T)-c\mathbf{M}(\partial T)
\end{equation}%
among all transport paths $T$ with $\partial T\preceq \nu -\mu $. Without loss of generality, we may assume that  $c_{\nu}=2c$ and $c_\mu=0$ in equation (\ref{eqn: h}).

For each $c$, by Theorem \ref{thm: existence}, the ROTPB($\mu ,\nu $) problem has a solution $%
T_{c}^{\ast }$ that minimizes $\mathbf{E}_{\alpha}^c$. When $c\leq 0$, by Proposition
\ref{prop: Tzero}, the problem has a unique solution $T_{c}^{\ast }=0$. Thus, in the following context, we only need to investigate $T_{c}^{\ast }$ for $c>0$.

\begin{proposition}\label{prop: upper_bound_T}
Suppose $\mu$ and $\nu$ are two \changetext{disjointly supported} measures on $X$ of equal mass, and $T_{c}^{\ast }\in Path(\mu _{c}^{\ast
},\nu _{c}^{\ast })$ is a solution to the
ROTPB($\mu ,\nu $) problem associated with $c>0$. Then, for any transport path $T\in Path(\mu, \nu)$,
\begin{equation*}
\mathbf{M}_{\alpha}(T)-\mathbf{M}_{\alpha}(T^*_c)\ge c(\mathbf{M}(\partial
T)-\mathbf{M}(\partial T_c^*))\ge 0,
\end{equation*}
and hence $\mathbf{M}_{\alpha}(T^*_c) \le d_{\alpha}(\mu, \nu)$. Moreover, \changetext{ define}
\begin{equation}
\label{eqn: s}
    s(\mu,\nu):=\inf\left\{||\mu-\tilde{\mu}||+||\nu-\tilde{\nu}||: \tilde{\mu}\changetext{\leq} \mu ,
\tilde{\nu}\changetext{\leq} \nu, \tilde{\nu}- \tilde{\mu}\neq \nu-\mu \right\}.
\end{equation}
\changetext{If $s(\mu,\nu)>0$} and $c>
\frac{d_{\alpha}(\mu, \nu)}{s(\mu,\nu)}$, then $T_c^*$ is an optimal transport path
in $Path(\mu, \nu)$.
\end{proposition}

\begin{proof}
Indeed, for any transport path $T\in Path(\mu ,\nu )$,
\begin{eqnarray*}
\mathbf{M}_{\alpha }(T)-\mathbf{M}_{\alpha }(T_{c}^{\ast }) &=&\left(\mathbf{E}_{\alpha}^{c}%
(T)+c\mathbf{M}(\partial T)\right)-(\mathbf{E}_{\alpha}^{c}(T_{c}^{\ast })+c\mathbf{M}(\partial
T_{c}^{\ast })) \\
&=&(\mathbf{E}_{\alpha}^{c}(T)-\mathbf{E}_{\alpha}^{c}(T_{c}^{\ast }))+c(\mathbf{M}(\partial T)-%
\mathbf{M}(\partial T_{c}^{\ast })) \\
&\geq &c(\mathbf{M}(\partial T)-\mathbf{M}(\partial T_{c}^{\ast }))\geq 0.
\end{eqnarray*}%
Also, when $s(\mu,\nu)>0$ and $c>\frac{d_{\alpha }(\mu ,\nu )}{s(\mu,\nu)}$, assume that $%
T_{c}^{\ast }$ is not an optimal transport path in $Path(\mu ,\nu )$. Since $T_c^*$ is a solution to the ROTPB($\mu ,\nu $) problem, it is an optimal transport path in $Path(\mu _{c}^{\ast
},\nu _{c}^{\ast })$. Thus, $\mu _{c}^{\ast
}\neq \mu$ and $\nu _{c}^{\ast }\neq \nu$.
Now, for any optimal transport path $T$ in $Path(\mu ,\nu )$, it follows that
\begin{eqnarray*}
d_{\alpha }(\mu ,\nu )&\geq& \mathbf{M}_{\alpha }(T)-\mathbf{M}_{\alpha
}(T_{c}^{\ast })\geq c(\mathbf{M}(\partial T)-\mathbf{M}(\partial
T_{c}^{\ast }))\\
&=&c(||\mu-\mu _{c}^{\ast
}||+||\nu-\nu _{c}^{\ast
}||)\geq cs(\mu,\nu),
\end{eqnarray*}%
a contradiction with the choice of $c$.
\end{proof}

Proposition \ref{prop: upper_bound_T} shows that the transportation cost $\mathbf{M}_{\alpha
}(T_{c}^{\ast })$ associated with the
solution $T_{c}^{\ast }$ is bounded from above. More interestingly, when the
parameter $c$, a measure of the profitability for relocating mass, is
sufficiently large, $T_{c}^{\ast }$ represents an optimal way of
transporting mass from $\mu $ to $\nu $. The intuition is that since the
transportation cost is bounded, a large enough $c$ would induce the planner
to move as much mass as possible. This argument can be further validated by
the following proposition, which derives an upper bound as well as the decay rate for the amount of
mass left unmoved by $T_{c}^{\ast }$.

\begin{proposition}
Suppose $||\mu ||=||\nu ||$, $c>0$, \changetext{ $1-\frac{1}{m}<\alpha<1$} and $T_{c}^{\ast }\in Path(\mu _{c}^{\ast
},\nu _{c}^{\ast })$ denotes the solution to the ROTPB($\mu ,\nu $) problem.
Then
\begin{equation}
||\mu -\mu _{c}^{\ast }||=||\nu -\nu _{c}^{\ast }||\leq \left( \frac{C_{m,\alpha}  diam(X)}{2c}%
\right) ^{\frac{1}{1-\alpha }},  \label{eqn: a-a_star}
\end{equation}
where $C_{m,\alpha}$ is the constant given in (\ref{eqn: C_m_alpha}).
\end{proposition}

\begin{proof}
Let $\tilde{T}\in Path(\mu -\mu _c^*, \nu-\nu_c^*)$ be an $\alpha$-optimal
transport path, and denote $T=T_c^*+\tilde{T}\in Path(\mu, \nu)$. By (\ref{eqn: d_alpha_upperbound}),
\begin{eqnarray*}
0&\le & \mathbf{E}_{\alpha}^{c}(T)- \mathbf{E}_{\alpha}^{c}(T_c^*) \\
&=&\left(\mathbf{M}_{\alpha}(T)-c \mathbf{M}(\partial T)\right)-\left(%
\mathbf{M}_{\alpha}(T_{c}^*)-c \mathbf{M}(\partial T_c^*)\right) \\
&=&\left(\mathbf{M}_{\alpha}(T)-\mathbf{M}_{\alpha}(T_{c}^*)\right)-c\left(%
\mathbf{M}(\partial T)- \mathbf{M}(\partial T_c^*)\right) \\
&\le & \mathbf{M}_{\alpha}(\tilde{T})-c(||\mu-\mu_c^*||+||\nu-\nu_c^*||) \\
&\le&C_{m,\alpha}  diam(X) ||\mu-\mu_c^*||^{\alpha}-2c||\mu-\mu_c^*||,
\end{eqnarray*}
which leads to inequality (\ref{eqn: a-a_star}).
\end{proof}

The next proposition characterizes the monotonicity properties of the
solution. Intuitively, as $c$ rises, the planner tends to move more mass
between sources and destinations, resulting in larger transportation costs.

\begin{proposition}
\label{prop: increasing}
Suppose $||\mu ||=||\nu ||$, $c>0$, \changetext{ $1-\frac{1}{m}<\alpha<1$} and $T_{c}^{\ast }\in Path(\mu _{c}^{\ast
},\nu _{c}^{\ast })$ denotes the solution to the ROTPB($\mu ,\nu $) problem. Then, as a function of $c\in \mathbb{R}$,

\begin{enumerate}
\item $\mathbf{E}_{\alpha}^{c}(T^*_c)$ is decreasing;

\item $\mathbf{M}_{\alpha}(T_{c}^*)$ is increasing with $\lim_{c\rightarrow
\infty}\mathbf{M}_{\alpha}(T_{c}^*)=d_{\alpha}(\mu, \nu)$;

\item $\mathbf{M}(\partial T_{c}^*)$ is increasing with $\lim_{c\rightarrow
\infty}\partial T_{c}^*=\nu - \mu$.
\end{enumerate}
\end{proposition}

\begin{proof}
Indeed, for any $c_1 < c_2$,
\begin{equation*}
\mathbf{E}_{\alpha}^{c_1}(T_{c_1}^*)=\mathbf{M}_{\alpha}(T_{c_1}^*)-c_1 \mathbf{M}%
(\partial T_{c_1}^*)\ge \mathbf{M}_{\alpha}(T_{c_1}^*)-c_2 \mathbf{M}%
(\partial T_{c_1}^*)=\mathbf{E}_{\alpha}^{c_2}(T_{c_1}^*)\ge \mathbf{E}_{\alpha}^{c_2}(T_{c_2}^*).
\end{equation*}

Also, the inequalities $\mathbf{E}_{\alpha}^{c_1}(T_{c_{1}}^{\ast })\leq \mathbf{E}_{\alpha}^{c_1}(T_{c_{2}}^{\ast })$ and $\mathbf{E}_{\alpha}^{c_2}(T_{c_{2}}^{\ast })\leq
\mathbf{E}_{\alpha}^{c_2}(T_{c_{1}}^{\ast })$ imply that
\begin{eqnarray*}
&&\mathbf{M}_{\alpha }(T_{c_{1}}^{\ast })-c_{1}\mathbf{M}(\partial
T_{c_{1}}^{\ast })\leq \mathbf{M}_{\alpha }(T_{c_{2}}^{\ast })-c_{1}\mathbf{M%
}(\partial T_{c_{2}}^{\ast }) \\
&&\mathbf{M}_{\alpha }(T_{c_{2}}^{\ast })-c_{2}\mathbf{M}(\partial
T_{c_{2}}^{\ast })\leq \mathbf{M}_{\alpha }(T_{c_{1}}^{\ast })-c_{2}\mathbf{M%
}(\partial T_{c_{1}}^{\ast }).
\end{eqnarray*}%
Rewriting them gives
\begin{equation*}
c_{2}\left( \mathbf{M}(\partial T_{c_{2}}^{\ast })-\mathbf{M}(\partial
T_{c_{1}}^{\ast })\right) \geq \mathbf{M}_{\alpha }(T_{c_{2}}^{\ast })-%
\mathbf{M}_{\alpha }(T_{c_{1}}^{\ast })\geq c_{1}\left( \mathbf{M}(\partial
T_{c_{2}}^{\ast })-\mathbf{M}(\partial T_{c_{1}}^{\ast })\right) .
\end{equation*}%
Since $c_{1}<c_{2}$, we have $\mathbf{M}(\partial T_{c_{2}}^{\ast })\geq
\mathbf{M}(\partial T_{c_{1}}^{\ast })$ and $\mathbf{M}_{\alpha
}(T_{c_{2}}^{\ast })\geq \mathbf{M}_{\alpha }(T_{c_{1}}^{\ast })$. This
shows that both $\mathbf{M}_{\alpha }(T_{c}^{\ast })$ and $\mathbf{M}%
(\partial T_{c}^{\ast })$ are increasing functions of $c$.

Moreover, by inequality (\ref{eqn: a-a_star}), $\lim_{c\rightarrow
\infty}\partial T_{c}^*=\lim_{c\rightarrow \infty}\nu_c^*-\mu _c^*=\nu - \mu$%
. Since $d_{\alpha}$ is a distance between measures of equal mass, \changetext{by Proposition \ref{prop: upper_bound_T},}
\begin{eqnarray*}
0 & \le & d_{\alpha}(\mu, \nu)-\lim_{c\rightarrow \infty}\mathbf{M}%
_{\alpha}(T_{c}^*) \\
&=& d_{\alpha}(\mu, \nu)-\lim_{c\rightarrow \infty}d_{\alpha}(\mu _c^*,
\nu_c^*) \\
&\le & \lim_{c\rightarrow \infty}\left( d_{\alpha}(\mu , \mu
_c^*)+d_{\alpha}(\nu, \nu_c^*) \right) =0.
\end{eqnarray*}
Thus, $d_{\alpha}(\mu, \nu)=\lim_{c\rightarrow \infty}\mathbf{M}%
_{\alpha}(T_{c}^*)$.
\end{proof}

\begin{theorem}
\label{thm: approximation}
Suppose $\mu$ and $\nu$ are two \changetext{disjointly supported} measures on $X$ of equal mass, \changetext{$1-\frac{1}{m}<\alpha<1$,} and let $T_{c}^{\ast }\in Path(\mu _{c}^{\ast
},\nu _{c}^{\ast })$ denote the solution to the ROTPB($\mu ,\nu $) problem corresponding to parameter $c$.
If for some sequence $\{c_n\}$ converging to $\infty$, the associated sequence $\{T_{c_{n}}^{\ast }\}$
is subsequentially convergent to $T$ as \changetext{rectifiable normal 1-currents with respect to flat convergence}, then $T$ is
an $\alpha $-optimal transport path from $\mu $ to $\nu $.
\end{theorem}

\begin{proof}
By the lower semi-continuity of $\mathbf{M}_{\alpha}$ and Proposition \ref{prop: increasing},
\begin{equation*}
\mathbf{M}_{\alpha}(T)\le \liminf_{n\rightarrow \infty}\mathbf{M}%
_{\alpha}(T_{c_n}^*)=d_{\alpha}(\mu, \nu).
\end{equation*}
Since $\partial T=\nu - \mu$, $T$ itself is also a transport path from $%
\mu $ to $\nu$, and it holds that $d_{\alpha}(\mu, \nu)\le \mathbf{M}%
_{\alpha}(T) $. As a result, $T$ is an optimal transport path.
\end{proof}

\begin{remark}
Theorem \ref{thm: approximation} provides a novel perspective for approximating an optimal transport path. In
light of this theorem, one can solve a sequence of ROTPB problems associated
with a monotonically increasing series of $\left\{ c_{n}\right\} $, and then
use the limit of their solutions to obtain the desired path. For small
values of $c_{n}$, the path $T_{c_{n}}^{\ast }$ is typically of simple
structure and thus relatively easy to solve. As $c_{n}$ rises, the planner
would start moving more mass through transport paths of increasing
complexity, which eventually converge to an optimal transport path from $\mu
$ to $\nu $. We leave exploration along this line to future research.
\end{remark}

\end{document}